\documentclass[]{article}

\usepackage{amsmath,amsthm,amssymb,mathrsfs,graphicx,subfigure,extarrows,float}
\usepackage[numbers]{natbib}
\usepackage{titlesec}
\usepackage[titletoc,toc,title]{appendix}
\usepackage{graphicx}
\usepackage{faktor}
\usepackage[a4paper, textheight=9.6in]{geometry}

\title{Non-Semisimple Planar Algebras from the\\ Representation Theory of $\bar{U}_{q}(\mathfrak{sl}_{2})$}
\author{Stephen Moore \\
	\multicolumn{1}{p{.8\textwidth}}{\centering\emph{Dept. of Mathematics, Ben Gurion University, Beer-Sheva, Israel,\\ stm862@gmail.com}}}
\date{2 August, 2018}

\newtheorem{thm}{Theorem}[section]
\newtheorem{prop}{Proposition}[section]

\newtheorem{conj}{Conjecture}[section]
\newtheorem{corr}{Corollary}[prop]

\begin{document}

\maketitle

\begin{abstract}
We describe the generators and prove a number of relations for the construction of a planar algebra from the restricted quantum group $\bar{U}_{q}(\mathfrak{sl}_{2})$. This is a diagrammatic description of $End_{\bar{U}_{q}(\mathfrak{sl}_{2})}(X^{\otimes n})$, where $X:=\mathcal{X}^{+}_{2}$ is a two dimensional $\bar{U}_{q}(\mathfrak{sl}_{2})$ module. 
\end{abstract}
\let\thefootnote\relax
\section{Introduction}
The restricted quantum group, $\bar{U}_{q}(\mathfrak{sl}_{2})$, for $q=e^{\frac{i\pi}{p}}$, $2\leq p\in\mathbb{N}$, is a finite dimensional quotient of $U_{q}(\mathfrak{sl}_{2})$ with non-semisimple representation theory, and was conjectured in \cite{FGST3} to have a representation category equivalent to the representation category of the $W_{p}$ logarithmic conformal field theory. An equivalence as abelian categories was proven in \cite{NaTs2}, however \cite{KoSa} showed that $\bar{U}_{q}(\mathfrak{sl}_{2})$ has modules whose tensor product does not commute, and so the category is not braided. This example does not appear in the subcategory generated by irreducible modules, which is the focus of our construction.\\

Planar algebras are a type of diagrammatic algebra first introduced as an axiomatization of the standard invariant of subfactors, and shown to have close relations to statistical mechanics and knot theory. The standard example of a planar algebra is the Temperley-Lieb algebra, which can be constructed from $U_{q}(\mathfrak{sl}_{2})$, as a description of $End_{U_{q}(\mathfrak{sl}_{2})}(X^{\otimes n})$, where $X$ is a two-dimensional irreducible module \cite{FrKhov, Martin}. We aim to generalize this construction to the restricted quantum group.\\

The $\bar{U}_{q}(\mathfrak{sl}_{2})$ planar algebra is a diagrammatic description of $End_{\bar{U}_{q}(\mathfrak{sl}_{2})}(X^{\otimes n})$, where $X$ is a two dimensional $\bar{U}_{q}(\mathfrak{sl}_{2})$ module. It was shown in \cite{GST} that for $n<2p-1$, this is isomorphic to the Temperley-Lieb algebra with parameter $q+q^{-1}$, and for $n\geq 2p-1$ contains extra generators, $\alpha_{i}$ and $\beta_{i}$, $1\leq i\leq n-2p+2$. We define these generators explicitly in terms of the $\bar{U}_{q}(\mathfrak{sl}_{2})$ action which allow us to prove combinatorially a number of relations on them. The main focus of the paper is the relations in Theorem \ref{thm} and their proofs. We also conjecture a formula for the dimensions of the planar algebra at the end of Section \ref{Uq}. Some of these relations were previously proven in \cite{GST}, however we include proofs here for completeness. Other relations in Theorem \ref{thm} generalize their results for $p=2$.\\

It was conjectured in \cite{GST} that $\alpha_{i}$ and $\beta_{i}$, along with the Temperley-Lieb algebra, generate $End_{\bar{U}_{q}(\mathfrak{sl}_{2})}(X^{\otimes n})$. We do not claim to have a proof of this, although we note that $\alpha_{1}\alpha_{p+1}$ and $\beta_{1}\beta_{p+1}$ give maps between the highest and lowest weight copies of $\mathcal{P}^{+}_{p}$ appearing in the decomposition of $X^{\otimes 3p-1}$, which is where we would expect any new generators to appear, and so lends evidence to support the conjecture. Further, we make the following conjecture:
\begin{conj}
The extension of the Temperley-Lieb algebra by the generators $\alpha_{i}$ and $\beta_{i}$ with the relations in Theorem \ref{thm} fully describe $End_{\bar{U}_{q}(\mathfrak{sl}_{2})}(X^{\otimes n})$.	
\end{conj}
The reasoning for this conjecture comes from considering diagrammatic descriptions of morphisms between modules, which can be considered as the reasoning behind many relations. Hence knowing the diagrammatic descriptions of all such morphisms suggests that the possibility of there being any unknown relations is less likely. These diagrammatic descriptions will be detailed in \cite{Me2}, and are also described in \cite{MeThesis}.\\\

The paper is outlined as follows: In Section $2$, we review the definition of $\bar{U}_{q}(\mathfrak{sl}_{2})$, along with the modules, fusion rules, and module homomorphisms relevant to our construction, as well as giving details on the dimensions of the planar algebra. In Section $3$, we introduce the planar algebra construction, and define the generators $\alpha$ and $\beta$, along with a number of relations on them. The proofs of the $\bar{U}_{q}(\mathfrak{sl}_{2})$ relations are given in Section $4$. A basis for $End_{\bar{U}_{q}(\mathfrak{sl}_{2})}(X^{\otimes 2p})$ is also given at the end of Section $4$. We record a number of combinatorial identities used in the proofs in the Appendix. For more detailed proofs, see \cite{MeThesis}.

\section{$\bar{U}_{q}(\mathfrak{sl}_{2})$ \label{Uq}}
For $q=e^{i\pi/p}$, $p\geq 2$, and $p\in \mathbb{N}$, the restricted quantum group $\bar{U}_{q}(\mathfrak{sl}_{2})$ over a field $\mathbb{K}$ is the Hopf algebra generated by $E,F,K$ subject to the relations:
\begin{align*}
KEK^{-1}&=q^{2}E& KFK^{-1}&=q^{-2}F& EF-FE&=\frac{K-K^{-1}}{q-q^{-1}}\\
E^{p}&=0& F^{p}&=0& K^{2p}&=1
\end{align*}
and coproduct $\Delta$, counit $\epsilon$, and antipode $S$:
\begin{align*}
\Delta:E&\mapsto E\otimes K+1\otimes E& F&\mapsto F\otimes 1+K^{-1}\otimes F& K&\mapsto K\otimes K\\
\epsilon:E&\mapsto 0& F&\mapsto 0& K&\mapsto 1\\
S:E&\mapsto-EK^{-1}& F&\mapsto-KF& K&\mapsto K^{-1}
\end{align*}
The modules for $\bar{U}_{q}(\mathfrak{sl}_{2})$ were given in \cite{Arike,GST,KoSa,Suter,Xiao}, and the modules relevant to our construction consist of the following:
$2p-2$ simple modules, $\mathcal{X}^{\pm}_{s}$, $1\leq s<p$, two simple projective modules $\mathcal{X}^{\pm}_{p}$ and $2p-2$ non-simple indecomposable projective modules $\mathcal{P}^{\pm}_{s}$. For the simple modules $\mathcal{X}^{\pm}_{s}$, $1\leq s\leq p$, they can be given in terms of a basis as $\{\nu_{n}^{s}\}_{n=0,...,s-1}$ with the action of $\bar{U}_{q}(\mathfrak{sl}_{2})$ given by:
\begin{align*}
K\nu_{n}&=\pm q^{s-1-2n}\nu_{n}\\
E\nu_{n}&=\pm[n][s-n]\nu_{n-1}\\
F\nu_{n}&=\nu_{n+1}
\end{align*}
where $\nu_{-1}=\nu_{s}=0$ and $[n]=\frac{q^{n}-q^{-n}}{q-q^{-1}}=q^{n-1}+q^{n-3}+...+q^{3-n}+q^{1-n}$. The projective modules $\mathcal{P}^{\pm}_{s}$, $1\leq s<p$, for a given choice of $p$, can be given in terms of the basis $\{a_{i}^{s,p},b_{i}^{s,p}\}_{0\leq i \leq s-1}\cup\{x_{j}^{s,p},y_{j}^{s,p}\}_{0\leq j\leq p-s-1}$. The action of $\bar{U}_{q}(\mathfrak{sl}_{2})$ is given by:
\begin{align*}
Ka_{i}&=\pm q^{s-1-2i}a_{i}& Kb_{i}&=\pm q^{s-1-2i}b_{i}\\
Kx_{j}&=\mp q^{p-s-1-2j}x_{j}& Ky_{j}&=\mp q^{p-s-1-2j}y_{j}\\
Ea_{i}&=\pm[i][s-i]a_{i-1}& Eb_{i}&=\pm[i][s-i]b_{i-1}+a_{i-1}& Eb_{0}&=x_{p-s-1}\\
Ex_{j}&=\mp[j][p-s-j]x_{j-1}& Ey_{j}&=\mp[j][p-s-j]y_{j-1}& Ey_{0}&=a_{s-1}\\
Fa_{i}&=a_{i+1}& Fb_{i}&=b_{i+1}& Fb_{s-1}&=y_{0}\\
Fx_{j}&=x_{j+1}& Fx_{p-s-1}&=a_{0}& Fy_{j}&=y_{j+1}
\end{align*}
where $x_{-1}=a_{-1}=a_{s}=y_{p-s}=0$, \cite{FGST2}.\\
\\
The maps between indecomposable modules can be summarized as follows:
\begin{itemize}
	\item $\dim\left( Hom(\mathcal{X}^{\pm}_{s},\mathcal{X}^{\pm}_{t})\right) =0$ for $s\neq t$ or $1$ for $s=t$, for $1\leq s,t\leq p$.
	\item $ Hom(\mathcal{X}^{\pm}_{s},\mathcal{X}^{\mp}_{t}) =0$ for $1\leq s,t\leq p$.
	\item $ \dim\left( Hom(\mathcal{P}^{\pm}_{s},\mathcal{X}^{\pm}_{t})\right) =0$ for $s\neq t$ or $1$ for $s=t$, for $1\leq s,t \leq p-1$.
	\item $ Hom(\mathcal{P}^{\pm}_{s},\mathcal{X}^{\mp}_{t}) =0$ for $1\leq s,t\leq p-1$.
	\item $\dim\left( Hom(\mathcal{P}^{\pm}_{s},\mathcal{P}^{\pm}_{t})\right) =0$ for $s\neq t$ or $2$ for $s=t$, $1\leq s,t\leq p-1$.
	\item $\dim\left( Hom(\mathcal{P}^{\pm}_{s},\mathcal{P}^{\mp}_{t})\right) =0$ for $s\neq p-t$ or $2$ for $s=p-t$, for $1\leq s,t\leq p-1$.
\end{itemize}
The non-zero, and non-identity parts of these maps can be given in terms of bases by:
\begin{align*}
\mathcal{P}^{\pm}_{s}\rightarrow\mathcal{X}^{\pm}_{s}:& b_{i}\mapsto \nu_{i}\\
\mathcal{X}^{\pm}_{s}\rightarrow\mathcal{P}^{\pm}_{s}:& \nu_{i}\mapsto a_{i}\\
\mathcal{P}^{\pm}_{s}\rightarrow\mathcal{P}^{\pm}_{s}:& b_{i}\mapsto a_{i}\\
\mathcal{P}^{\pm}_{s}\rightarrow\mathcal{P}^{\mp}_{s}:& b_{i}\mapsto f_{1}\tilde{x}_{i}+f_{2}\tilde{y}_{i}\\
& x_{j}\mapsto f_{2}\tilde{a}_{j}\\
& y_{j}\mapsto f_{1}\tilde{a}_{j}
\end{align*}
where $f_{1},f_{2}\in\mathbb{K}$, and we denote the elements of $\mathcal{P}^{\mp}_{s}$ with $\sim$.\\
\\
From now on, we denote the module $\mathcal{X}^{+}_{2}$ by $X$. The fusion rules for $\bar{U}_{q}(\mathfrak{sl}_{2})$ modules were given in \cite{KoSa,TsuWo}. The fusion rules relevant to our planar algebra construction are:
\begin{align*}
&\mathcal{X}^{i}_{1}\otimes\mathcal{X}^{j}_{s}\simeq\mathcal{X}^{j}_{s}\otimes\mathcal{X}^{i}_{1}\simeq\mathcal{X}^{ij}_{s}, \:\:\: i,j\in\{+,-\}\\
&\mathcal{X}^{i}_{1}\otimes\mathcal{P}^{j}_{s}\simeq\mathcal{P}^{j}_{s}\otimes\mathcal{X}^{i}_{1}\simeq\mathcal{P}^{ij}_{s}\\
& X\otimes\mathcal{X}^{\pm}_{t}\simeq\mathcal{X}^{\pm}_{t}\otimes X\simeq\mathcal{X}^{\pm}_{t-1}\oplus\mathcal{X}^{\pm}_{t+1}, \:\:\: 2\leq t\leq p-1\\
& X\otimes\mathcal{X}^{\pm}_{p}\simeq \mathcal{X}^{\pm}_{p}\otimes X\simeq \mathcal{P}^{\pm}_{p-1}\\
& X\otimes\mathcal{P}^{\pm}_{p-1}\simeq \mathcal{P}^{\pm}_{p-1}\otimes X\simeq \mathcal{P}^{\pm}_{p-2}\oplus 2\mathcal{X}^{\pm}_{p}\\
& X\otimes\mathcal{P}^{\pm}_{u}\simeq\mathcal{P}^{\pm}_{u}\otimes X\simeq\mathcal{P}^{\pm}_{u-1}\oplus\mathcal{P}^{\pm}_{u+1}, \:\:\: 2\leq u\leq p-2\\
& X\otimes\mathcal{P}^{\pm}_{1}\simeq\mathcal{P}^{\pm}_{1}\otimes X\simeq\mathcal{P}^{\pm}_{2}\oplus 2\mathcal{X}^{\mp}_{p}
\end{align*}
From this, it follows that $\mathcal{X}^{+}_{s}$ first appears in the decomposition of $X^{\otimes (s-1)}$, $\mathcal{P}^{+}_{t}$ first appears in the decomposition of $X^{\otimes(2p-t-1)}$, $\mathcal{X}^{-}_{p}$ in $X^{\otimes(2p-1)}$, and $\mathcal{P}^{-}_{u}$ in $X^{\otimes(3p-u-1)}$. $\mathcal{X}^{-}_{v}$, $1\leq v\leq p-1$, does not appear at all in the decomposition of $X^{\otimes n}$.\\
\\
Using these fusion rules, along with the maps between modules, we can describe the dimensions of $End(X^{\otimes n})$, which we denote by $D_{n}$. We denote by $M_{i,n}$ the multiplicity of $\mathcal{X}^{+}_{i}$ in the decomposition of $X^{\otimes n}$, $1\leq i\leq p$. $M_{2p-j,n}$ denotes the multiplicity of $\mathcal{P}^{+}_{j}$, $1\leq j\leq p-1$, $M_{2p,n}$ the multiplicity of $\mathcal{X}^{-}_{p}$, and $M_{3p-j,n}$ the multiplicity of $\mathcal{P}^{-}_{j}$. For a given choice of $p$, we then have:
\begin{align*}
D_{n}&=M_{2p,n}^{2}+\sum\limits_{i=1}^{p}M_{i,n}^{2}+\sum\limits_{j=1}^{p-1}(2M_{p+j,n}^{2}+2M_{2p+j,n}^{2}+2M_{p-j,n}M_{p+j,n}+4M_{p+j,n}M_{3p-j,n})
\end{align*}
By use of the fusion rules to give module multiplicities in the decomposition of $X^{\otimes n}$, the following was proven in section 4.5 of \cite{GST}:
\begin{prop}\label{prop1}
For $n<2p-1$, $D_{n}$ is equal to the Catalan number, $C_{n}:=\frac{1}{(n+1)}\left(\begin{array}{c}2n\\n\end{array}\right)$, and $D_{2p-1}=C_{2p-1}+3$.
\end{prop} 
Another proof is contained in Section 4.3 of \cite{MeThesis}. In general, the dimensions $D_{n}$ can be calculated from the $\bar{U}_{q}(\mathfrak{sl}_{2})$ fusion rules and homomorphisms, or by use of results in Section 4.9 of \cite{GST}. Alternatively, we have the following:\\

Let $\left\{\begin{array}{c}a\\b\end{array}\right\}:=\left(\begin{array}{c}a\\b\end{array}\right)-\left(\begin{array}{c}a\\b-1\end{array}\right)$, and\\ $G_{n,i}:=\sum\limits_{j=0}^{\lfloor\frac{i}{2}\rfloor}2\left\{\begin{array}{c}n\\i+1-j\end{array}\right\}\left\{\begin{array}{c}n\\j\end{array}\right\}+\left(\lfloor\frac{(i+1)}{2}\rfloor-\lfloor\frac{i}{2}\rfloor\right)\left\{\begin{array}{c}n\\\lfloor\frac{i}{2}\rfloor+1\end{array}\right\}^{2}$	
\begin{conj}
\begin{align*}
D_{n}=C_{n}+\sum\limits_{j=0}^{\lfloor\frac{n}{p}\rfloor}(n+1)(n+3)G_{n,n-(j+2)p}, \:\text{ for all $n$.}
\end{align*}
\end{conj}
It can be verified numerically for small values of $n$ that this agrees with the formula for $p=2$ given in Section 2.4.3 of \cite{GST}.\\ \\
The module $X:=\mathcal{X}^{+}_{2}$ has basis $\{\nu_{0},\nu_{1}\}$, with $\bar{U}_{q}(\mathfrak{sl}_{2})$ action:
\begin{align*}
K(\nu_{0})&=q\nu_{0} & E(\nu_{0})&=0 & F(\nu_{0})&=\nu_{1}\\
K(\nu_{1})&=q^{-1}\nu_{1} & E(\nu_{1})&=\nu_{0} & F(\nu_{1})&=0
\end{align*}
The action of $\bar{U}_{q}(\mathfrak{sl}_{2})$ on $X^{\otimes n}$ is given by use of the coproduct.\\
\\
We denote by $\rho_{i_{1},...,i_{n},z}$ the element of $X^{\otimes z}$ with $\nu_{1}$ at positions $i_{1},...,i_{n}$, and $\nu_{0}$ elsewhere. We also occasionally omit the $\otimes$ sign, and combine indices. For example, $\rho_{1,3,5}=\nu_{1}\otimes\nu_{0}\otimes\nu_{1}\otimes\nu_{0}\otimes\nu_{0}=\nu_{10100}$. The elements of $X^{\otimes z}$ can be described in terms of the $K$-action on them. For $x\in X^{\otimes z}$, with $K(x)=\lambda x$, $\lambda\in\mathbb{K}$, we call $\lambda$ the weight of $x$. Alternatively for basis elements we can write this as $K(\rho_{i_{1},...,i_{n},z})=q^{z-2n}x$, and refer to $n$ also as the weight. $X^{\otimes z}$ will then have the set of weights $\{q^{z},q^{z-2},...,q^{2-z},q^{-z}\}$. Denoting the set of elements of $X^{\otimes z}$ with weight $q^{z-2n}$ by $X_{n,z}$, we have $X^{\otimes z}=\bigcup\limits_{i=0}^{z}X_{i,z}$. The weight spaces $X_{0,z}$ $X_{z,z}$ both have a single element, which we denote by $x_{0,z}:=(\nu_{0})^{\otimes z}$, $x_{z,z}:=(\nu_{1})^{\otimes z}$ respectively, and occasionally drop the second index if the context is clear. We have $\rho_{i_{1},...,i_{n},z}\in X_{n,z}$. We record a number of combinatorial relations involving $\bar{U}_{q}(\mathfrak{sl}_{2})$ and its action on $X^{\otimes z}$ in the Appendix.

\section{The $\bar{U}_{q}(\mathfrak{sl}_{2})$ Planar Algebra}
For detailed introductions to planar algebras, see \cite{JonesP1, MPS}. Our construction of the $\bar{U}_{q}(\mathfrak{sl}_{2})$ planar algebra is a diagrammatic description of $End\big(X^{\otimes n}\big)$, or more generally, $Hom(X^{\otimes n},X^{\otimes m})$, similar to the constructions of \cite{EvansPugh2,FrKhov,Kuper,Morrison}. It was shown in \cite{GST} that for $n<2p-1$, this is isomorphic to the Temperley-Lieb algebra on $n$ points with parameter $\delta=q+q^{-1}$, and for $n\geq 2p-1$,  $End\big(X^{\otimes n}\big)$ contains extra generators, $\alpha_{i},\beta_{i}$, $1\leq i\leq n-2p+1$, which are described in Section \ref{gen}. We then prove a number of relations on these generators.\\

Diagrammatically, $Hom(X^{\otimes n},X^{\otimes m})$ is represented by a box with $n$ points along the top, and $m$ points along the bottom, where zero points represent a map to/from $\mathcal{X}^{+}_{1}\simeq\mathbb{K}$. Each point on the box is connected to a string, up to isotopy, with strings not allowed to intersect. Removing a closed loop from a diagram corresponds to multiplying by $\delta\in\mathbb{K}$. Multiplication of maps is given by adjoining a diagram below another diagram, and smoothing strings. Tensor products are given by by adjoining diagrams side by side. The identity element in $End(X^{\otimes n})$ is given by $n$ vertical strings. The identity map $\mathbb{K}\rightarrow\mathbb{K}$ is given by an empty box. We often use a single thick string to represent multiple parallel strings.\\
\\
The \textit{Temperley-Lieb Algebra}, $TL_{n}(\delta)$, is the algebra generated by $\{1,\mathbf{e}_{1},...,\mathbf{e}_{n-1}\}$, $\delta\in\mathbb{K}$, with relations:
\begin{align*}
\mathbf{e}_{i}^{2}=&\delta \mathbf{e}_{i}\\
\mathbf{e}_{i}\mathbf{e}_{i\pm 1}\mathbf{e}_{i}=& \mathbf{e}_{i}\\
\mathbf{e}_{i}\mathbf{e}_{j}=& \mathbf{e}_{j}\mathbf{e}_{i}, \:\:\: \lvert i-j\rvert >1
\end{align*}
The \textit{Jones-Wenzl Projections}, \cite{Wenzl}, in the Temperley-Lieb algebra are defined inductively by:
\begin{align*}
\mathfrak{f}_{1}:=& 1\\
\mathfrak{f}_{n+1}:=& \mathfrak{f}_{n}\otimes 1 -\frac{[n]}{[n+1]}\mathfrak{f}_{n}\mathbf{e}_{n}\mathfrak{f}_{n}
\end{align*}
The following Proposition comes from Section 4.4 of \cite{GST}:
\begin{prop}
$End_{\bar{U}_{q}(\mathfrak{sl}_{2})}(X^{\otimes n})\simeq TL_{n}(q+q^{-1})$, for $n<2p-1$.
\end{prop}
\begin{proof}
The generators $\mathbf{e}_{i}$ correspond to a map $e:X\otimes X\rightarrow \mathbb{K}\rightarrow X\otimes X$, $e\in End(X^{\otimes n})$, acting on the $i$th and $(i+1)$th positions in $X^{\otimes n}$. Explicitly, this is given by:
\begin{align*}
e:&\nu_{0}\otimes\nu_{0}\mapsto 0\\
e:&\nu_{1}\otimes\nu_{1}\mapsto 0\\
e:&\nu_{0}\otimes\nu_{1}\mapsto q\nu_{0}\otimes\nu_{1}-\nu_{1}\otimes\nu_{0}\\
e:&\nu_{1}\otimes\nu_{0}\mapsto q^{-1}\nu_{1}\otimes\nu_{0}-\nu_{0}\otimes\nu_{1}\\
\mathbf{e}_{i}\simeq& 1^{\otimes (i-1)}\otimes e\otimes 1^{\otimes(n-i-1)}
\end{align*}
This correspondence with the generators $\mathbf{e}_{i}$ clearly extends to all elements of $TL_{n}$, and so gives an injective map from $End(X^{\otimes n})$ to $TL_{n}$. (To see the injectivity, consider, for example, $\mathbf{e}_{1}\mathbf{e}_{2}$ and $\mathbf{e}_{2}\mathbf{e}_{1}$ acting on $\nu_{011}$). From Proposition \ref{prop1}, we have $\dim\left( End(X^{\otimes n})\right)=\dim\left( TL_{n}\right)$. Hence\\ $End_{\bar{U}_{q}(\mathfrak{sl}_{2})}(X^{\otimes n})\simeq TL_{n}(q+q^{-1})$.
\end{proof}
The generator $e$ can be split into two maps; $\cup:X\otimes X\rightarrow\mathcal{X}^{+}_{1}$ and $\cap:\mathcal{X}^{+}_{1}\rightarrow X\otimes X$. In terms of $\nu_{ij}$, they are given as:
\begin{align*}
\cup(\nu_{00})=&\cup(\nu_{11})=0\\
\cup(\nu_{01})=&-q\nu\\
\cup(\nu_{10})=& \nu\\
\cap(\nu)=& q^{-1}\nu_{10}-\nu_{01}
\end{align*}
where $\mathcal{X}^{+}_{1}$ has basis $\{\nu\}$. These extend to maps $\cup_{i}$, $\cap_{i}$, acting on the $i$th and $(i+1)$th points such that $\cap_{i}\cup_{i}\simeq\mathbf{e}_{i}$. Note that technically we need to add the condition that moving a cup or cap rightwards multiplies it by $-1$ to ensure rigidity, however this is not required for our results. The following corollary is a well known result:
\begin{corr}
The projection $X^{\otimes n}\rightarrow \mathcal{X}^{+}_{n+1}\rightarrow X^{\otimes n}$, is given by the Jones-Wenzl projection, $\mathfrak{f}_{n}$.\\
\end{corr}
Expicitly, this projection is given by $\rho_{i_{1},...,i_{k},n}\mapsto q^{\big(kn-\frac{1}{2}(k^{2}-k)-(\sum\limits_{j=1}^{k}i_{j})\big)}\frac{([n-k]!)}{([n]!)}F^{k}x_{0,n}$. This can then be shown to satisfy the Jones-Wenzl recursion relation.\\
\\
We saw in Proposition \ref{prop1} that $\dim\left( End(X^{\otimes (2p-1)})\right)=C_{2p-1}+3$, and so $End(X^{\otimes (2p-1)})$ must contain extra generators. Our main focus now is to describe these generators and their relations.\\
\\
\subsection{The Generators, $\alpha$ and $\beta$.\label{gen}}
We define the following maps on $X^{\otimes(2p-1)}$:
\begin{align*}
\alpha(\rho_{i_{1},...,i_{k},2p-1})&:=q^{\big(k(2p-1)-\frac{1}{2}(k^{2}-k)-(\sum\limits_{j=1}^{k}i_{j})\big)}([k]!)E^{p-k-1}x_{2p-1}\\
\beta(\rho_{i_{1},...,i_{k},2p-1})&:=q^{\big(k(2p-1)-\frac{1}{2}(k^{2}-k)-(\sum\limits_{j=1}^{k}i_{j})\big)}([2p-1-k]!)F^{k-p}x_{0}
\end{align*}
where we take $E^{-1}=0$, $F^{-1}=0$. We often use a second, simplified definition; let $x\in X_{k,2p-1}$, and $e_{x},f_{x}\in\mathbb{K}$, such that $E^{k}x=e_{x}x_{0}$, $F^{2p-1-k}x=f_{x}x_{2p-1}$. Then we have:
\begin{align*}
\alpha(x)=& e_{x}E^{p-k-1}x_{2p-1}\\
\beta(x)=& f_{x}F^{k-p}x_{0}
\end{align*}
These are equivalent to the generators defined in \cite{GST}, up to a constant. In terms of weight spaces, the generators act as:
\begin{align*}
\alpha:& X_{k,2p-1}\rightarrow X_{k+p,2p-1}\\
\beta:& X_{k,2p-1}\rightarrow X_{k-p,2p-1} 
\end{align*} Then $\alpha$ is zero for $k\geq p$ and $\beta$ is zero for $k< p$. Hence $\alpha^{2}=\beta^{2}=0$. From their action on weight spaces, it is clear that $\alpha,\beta\notin TL_{2p-1}$. Let $\gamma:=(-1)^{p-1}([p-1]!)^2$. The following properties of $\alpha$ and $\beta$ were given in \cite{GST}:
\begin{prop}
\begin{align*}
End(X^{\otimes(2p-1)})&\simeq TL_{2p-1}(q+q^{-1})\oplus\mathbb{K}\alpha\oplus\mathbb{K}\beta\oplus\mathbb{K}\alpha\beta\\
\alpha\beta+\beta\alpha&=\gamma \mathfrak{f}_{2p-1}\\
\alpha\beta\alpha&=\gamma\alpha\\
\beta\alpha\beta&=\gamma\beta
\end{align*}
$\gamma^{-1}\alpha\beta$ and $\gamma^{-1}\beta\alpha$ give projections onto the two copies of $\mathcal{X}^{-}_{p}$ appearing in the decomposition of $X^{\otimes(2p-1)}$, $\alpha$ and $\beta$ are then maps between these two modules. \label{prop 3}
\end{prop}
See Sections \ref{subsect 2}, \ref{subsect 3}, and \ref{subsect 6} for proofs.\\
\\
The maps $\alpha$ and $\beta$ can be given diagrammatically as boxes with $2p-1$ points along the top and bottom (or a single thick string to represent multiple strings). They then form the generators of the $\bar{U}_{q}(\mathfrak{sl}_{2})$ planar algebra.\\
We denote by $\alpha_{i}$, $\beta_{i}$ the elements $1^{\otimes (i-1)}\otimes\alpha\otimes 1^{\otimes(n-2p-i+2)},1^{\otimes (i-1)}\otimes\beta\otimes 1^{\otimes(n-2p-i+2)}\in End(X^{\otimes n})$.\\
\begin{thm}\label{thm}
The generators, $\alpha_{i}$ and $\beta_{i}$, satisfy the following properties:
\begin{align}
\alpha^{2}&=\beta^{2}=0 \label{eq:1}\\
 \alpha\beta\alpha&=\gamma\alpha \label{eq:2}\\
 \beta\alpha\beta&=\gamma\beta \label{eq:3}\\
 \gamma&=(-1)^{p-1}([p-1]!)^{2} \nonumber \\
 \alpha_{i}\alpha_{j}=\alpha_{j}\alpha_{i}&=\beta_{i}\beta_{j}=\beta_{j}\beta_{i}=0, \:\: \lvert i-j\rvert< p \label{eq:4}\\
 \alpha_{i}\alpha_{i+p}&=\alpha_{i+p}\alpha_{i} \label{eq:5}\\
 \beta_{i}\beta_{i+p}&=\beta_{i+p}\beta_{i} \label{eq:6}\\
 \alpha\beta+\beta\alpha&=\gamma \mathfrak{f}_{2p-1} \label{eq:7}
 \end{align}
Denote by $R_{n}$ the (clockwise) (n,n)-point annular rotation tangle. We then have:
 \begin{align}
 \alpha\cap_{i}=\cup_{i}\alpha&=\beta\cap_{i}=\cup_{i}\beta=0, \:\: 1\leq i\leq 2p-2 \label{eq:8}\\
 \alpha_{i+1}\cap_{i}&=\alpha_{i}\cap_{i+2p-2} \label{eq:9}\\
  \beta_{i+1}\cap_{i}&=\beta_{i}\cap_{i+2p-2} \label{eq:10}\\
 \cup_{i}\alpha_{i+1}&=\cup_{i+2p-2}\alpha_{i} \label{eq:11}\\
 \cup_{i}\beta_{i+1}&=\cup_{i+2p-2}\beta_{i} \label{eq:12}\\
 R_{4p-2}(\alpha)&=\alpha \label{eq:13}\\
 R_{4p-2}(\beta)&=\beta \label{eq:14}\\
 \sum\limits_{i=0}^{4p-1} k_{i}R^{i}_{4p}(\alpha\otimes 1)&=0 \label{eq:15}\\
 \sum\limits_{i=0}^{4p-1} k_{i}R^{i}_{4p}(\beta\otimes 1)&=0 \label{eq:16}
 \end{align}
 where $k_{i}=(-1)^{i}[i-2]k_{1}+(-1)^{i}[i-1]k_{2}$, for arbitrary $k_{1},k_{2}\in\mathbb{K}$.\\
 \\
 Diagrammatically, these relations are:
 \begin{figure}[H]
 	\centering
 	\includegraphics[width=0.3\linewidth]{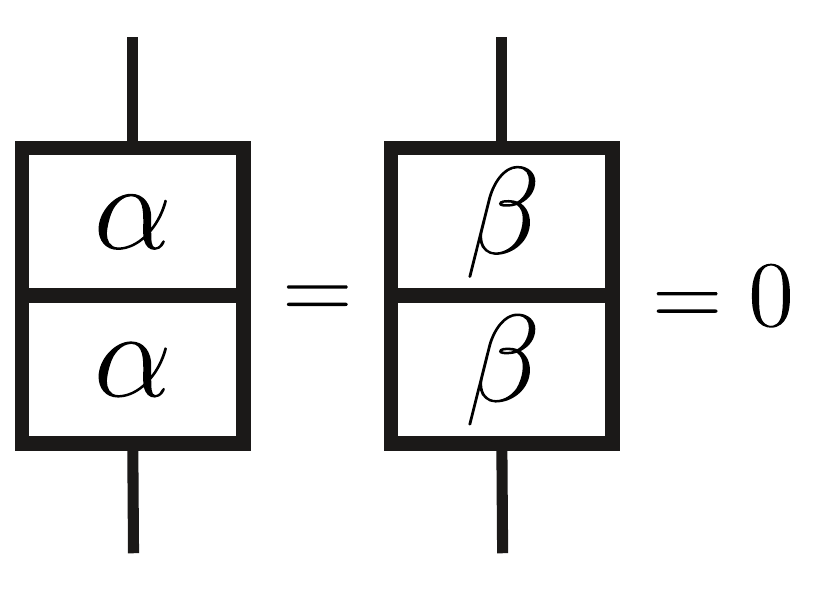}
 	\caption{Relation \ref{eq:1}}
 \end{figure}
 \begin{figure}[H]
 	\centering
 	\includegraphics[width=0.5\linewidth]{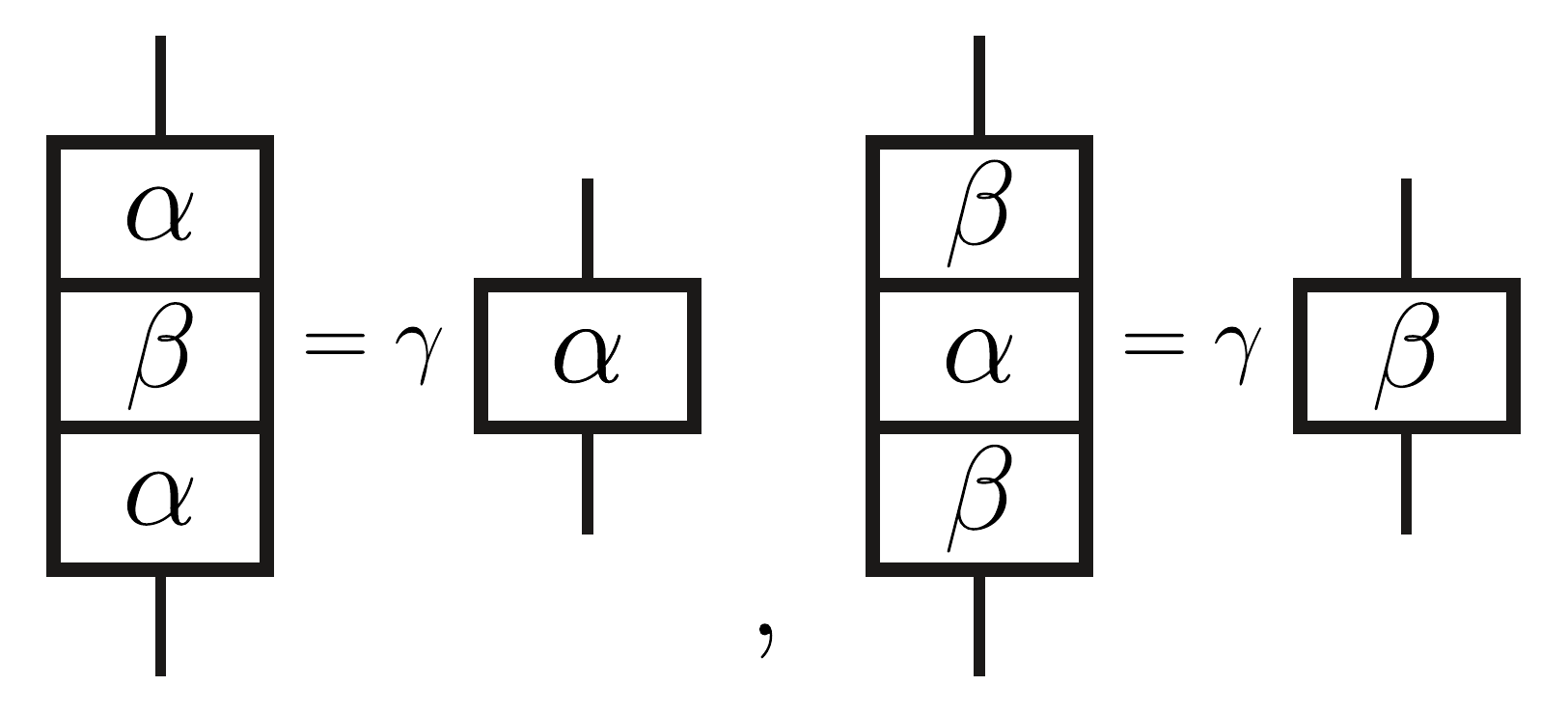}
 	\caption{Relation \ref{eq:2} and \ref{eq:3}}
 \end{figure}
 \begin{figure}[H]
 	\centering
 	\includegraphics[width=0.5\linewidth]{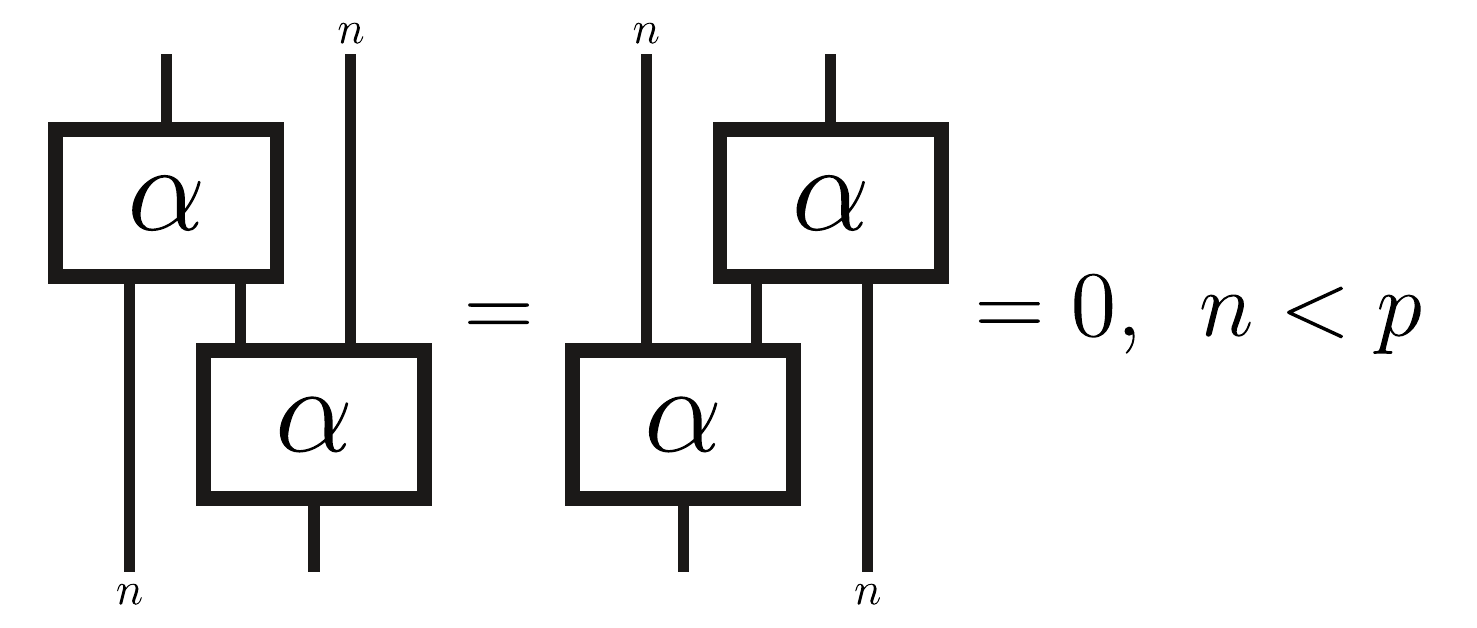}
 	\caption{Relation \ref{eq:4}}
 \end{figure}
 \begin{figure}[H]
 	\centering
 	\includegraphics[width=0.5\linewidth]{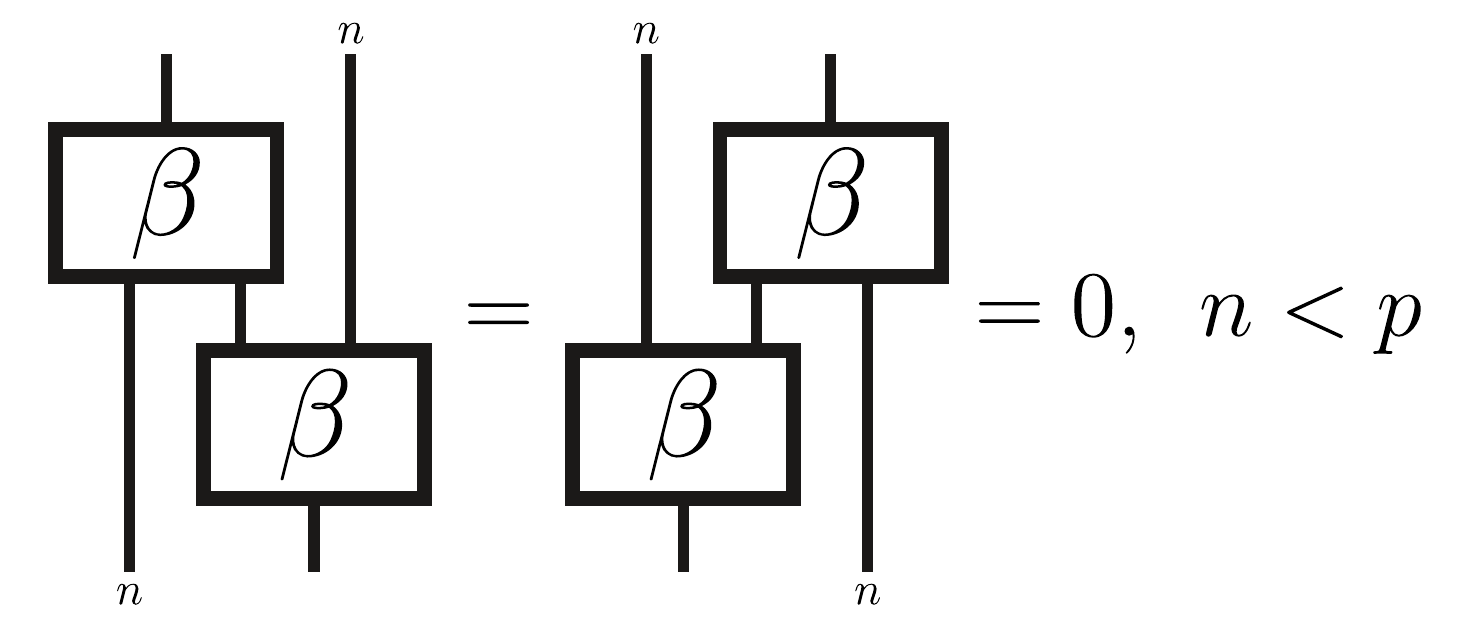}
 	\caption{Relation \ref{eq:4}}
 \end{figure}
 \begin{figure}[H]
 	\centering
 	\includegraphics[width=0.8\linewidth]{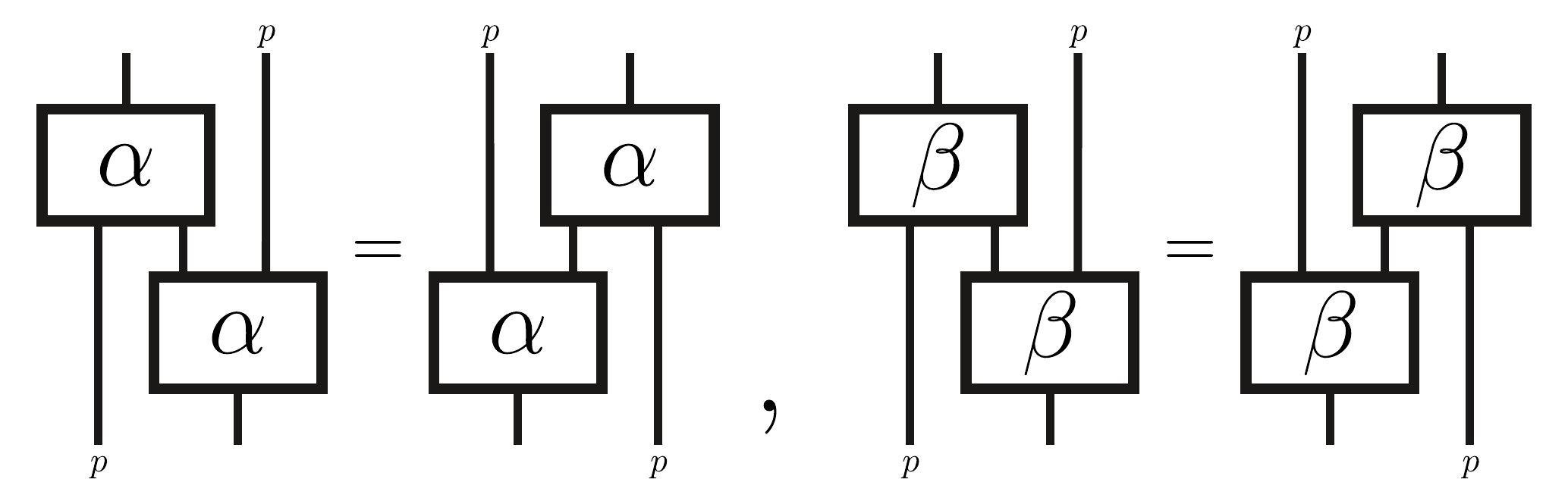}
 	\caption{Relation \ref{eq:5} and \ref{eq:6}}
 \end{figure}
 \begin{figure}[H]
 	\centering
 	\includegraphics[width=0.4\linewidth]{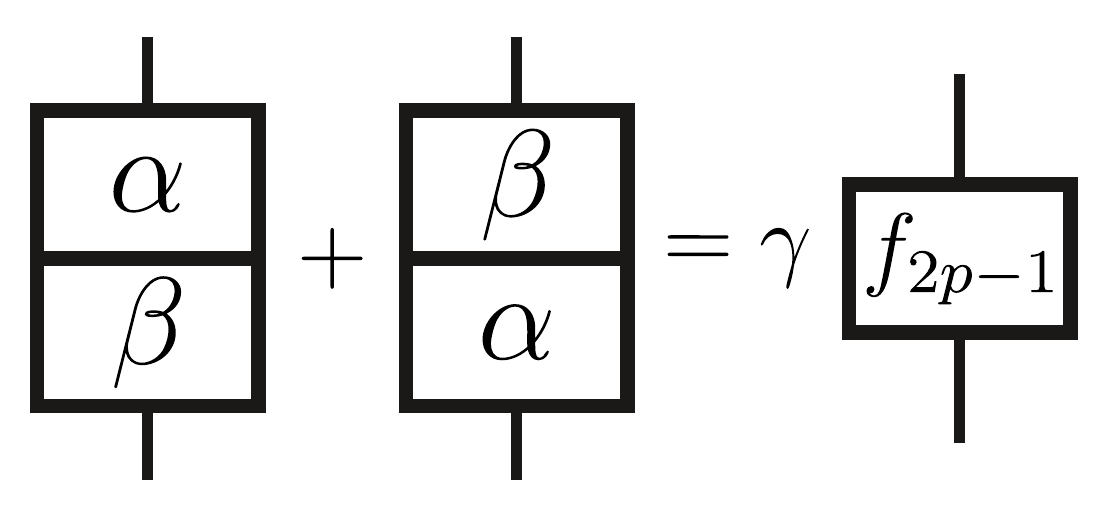}
 	\caption{Relation \ref{eq:7}}
 \end{figure}
 \begin{figure}[H]
 	\centering
 	\includegraphics[width=0.7\linewidth]{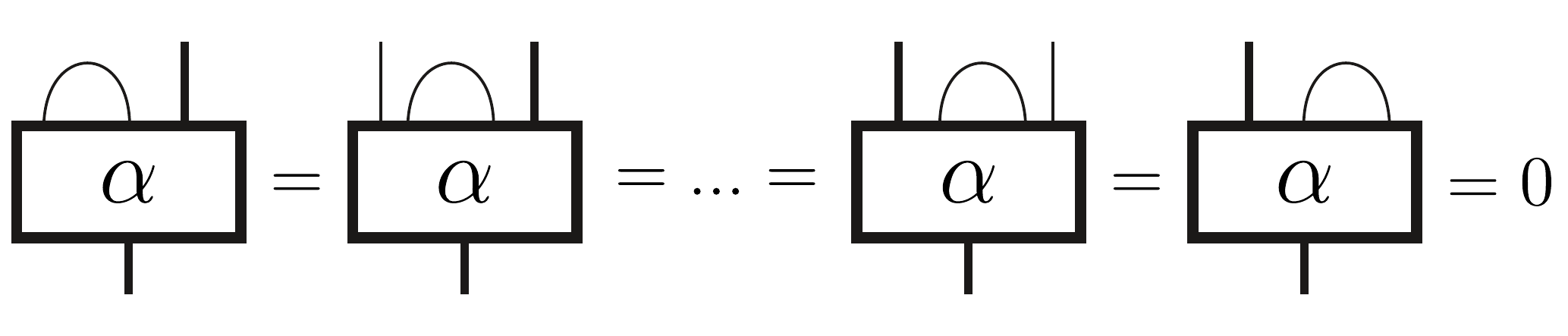}
 	\includegraphics[width=0.7\linewidth]{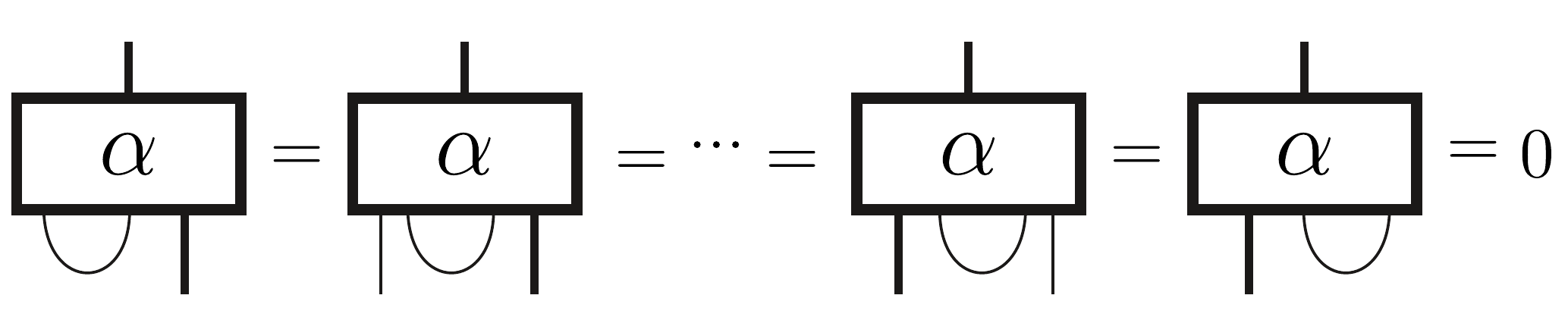}
 	\caption{Relation \ref{eq:8}}
 \end{figure}
 \begin{figure}[H]
 	\centering
 	\includegraphics[width=0.7\linewidth]{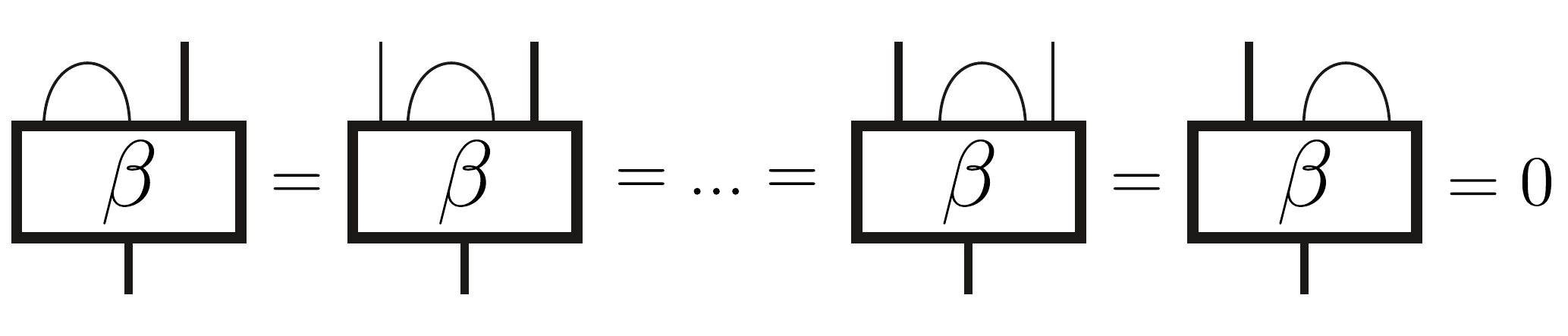}
 	\includegraphics[width=0.7\linewidth]{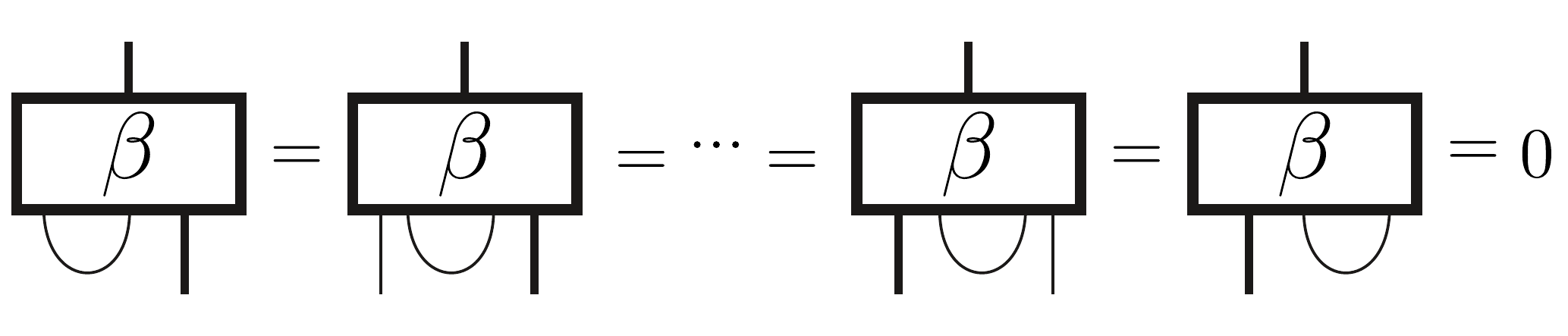}
 	\caption{Relation \ref{eq:8}}
 \end{figure}
 \begin{figure}[H]
 	\centering
 	\includegraphics[width=0.7\linewidth]{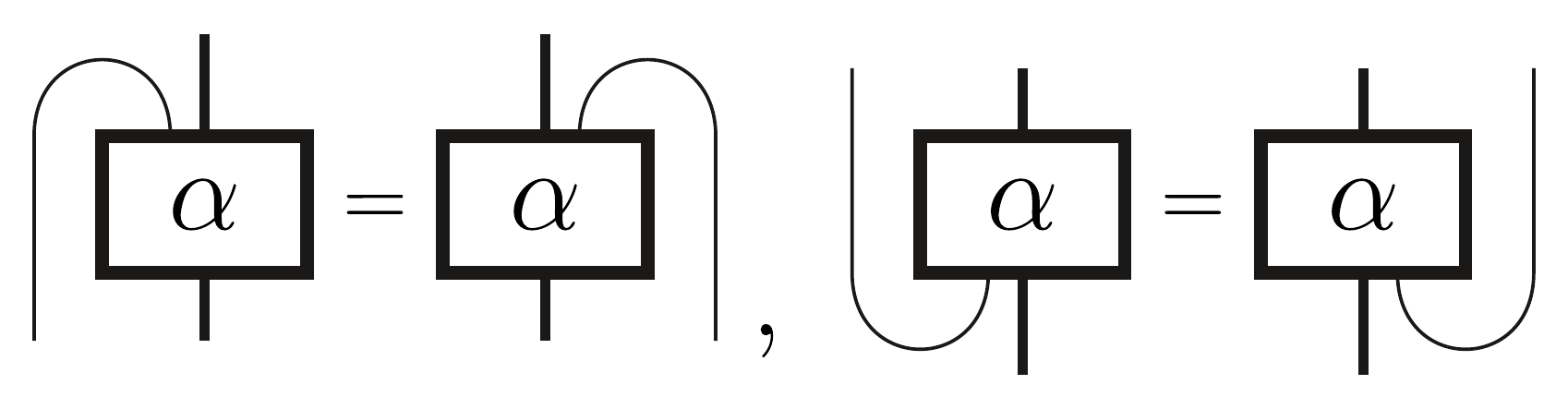}
 	\caption{Relation \ref{eq:9} and \ref{eq:11}}
 \end{figure}
 \begin{figure}[H]
 	\centering
 	\includegraphics[width=0.7\linewidth]{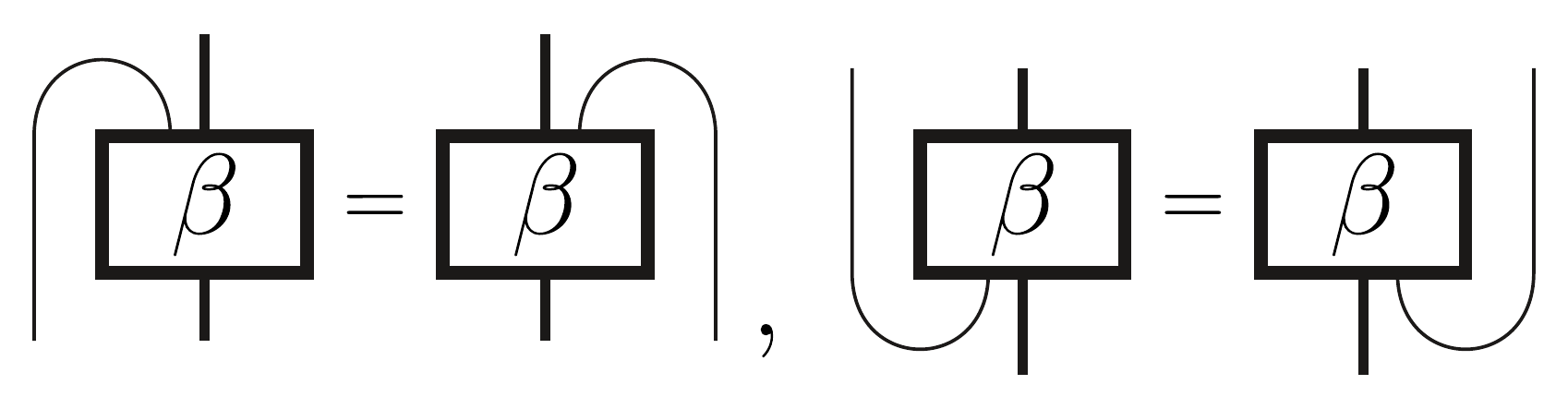}
 	\caption{Relation \ref{eq:10} and \ref{eq:12}}
 \end{figure}
 \begin{figure}[H]
 	\centering
 	\includegraphics[width=0.7\linewidth]{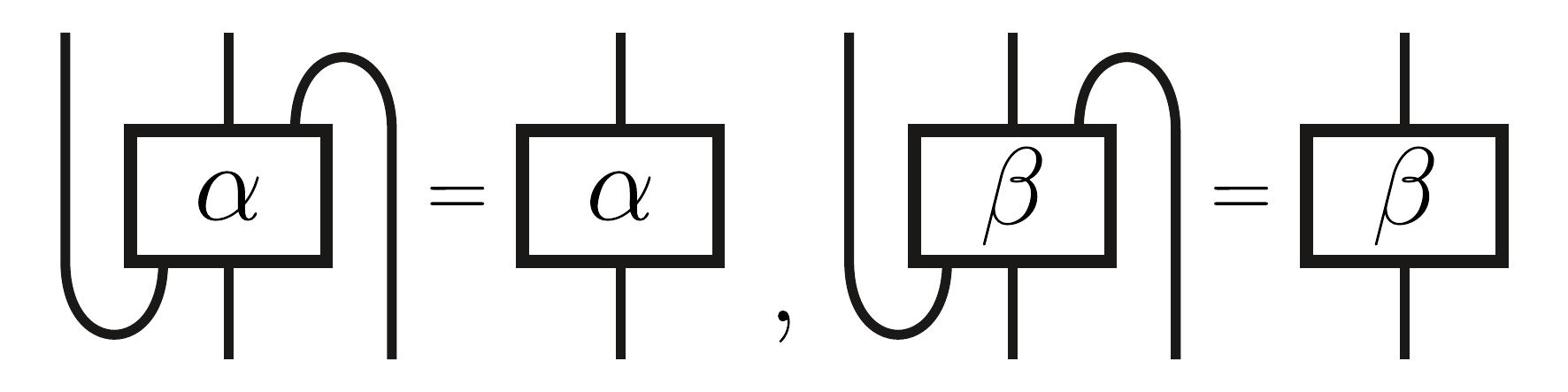}
 	\caption{Relation \ref{eq:13} and \ref{eq:14}}
 \end{figure}
 \begin{figure}[H]
 	\centering
 	\includegraphics[width=0.9\linewidth]{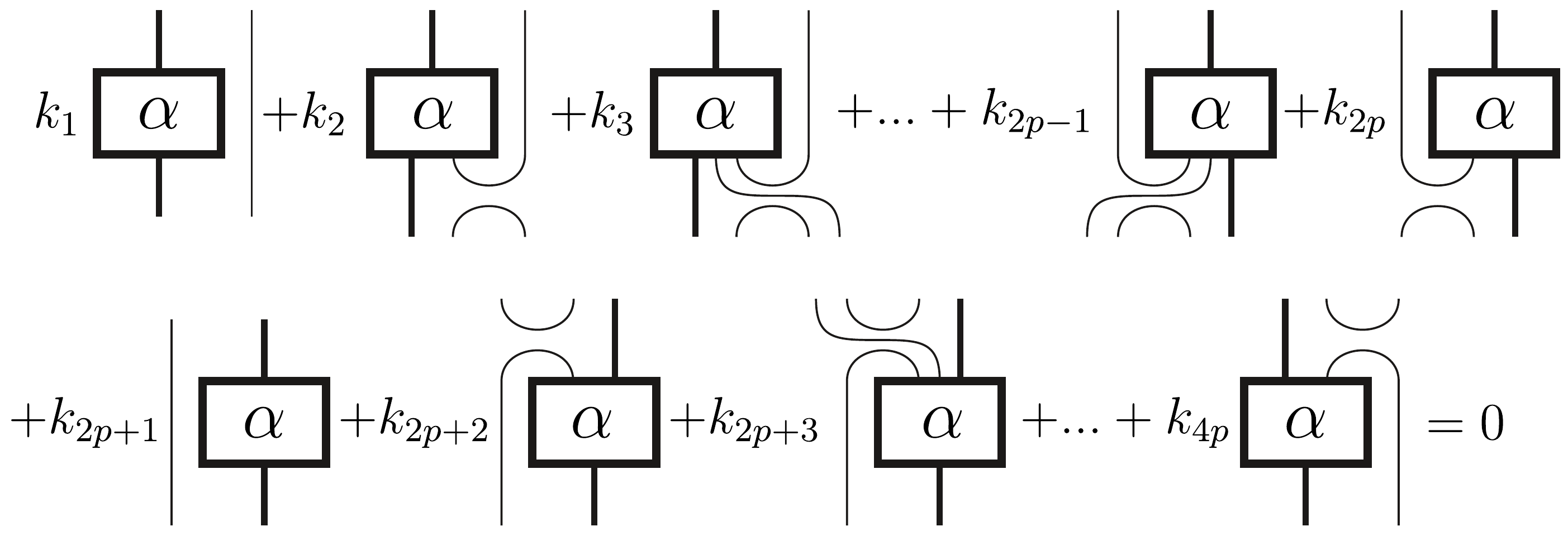}
 	\caption{Relation \ref{eq:15}} \label{fig:rot}
 \end{figure}
 \begin{figure}[H]
 	\centering
 	\includegraphics[width=0.9\linewidth]{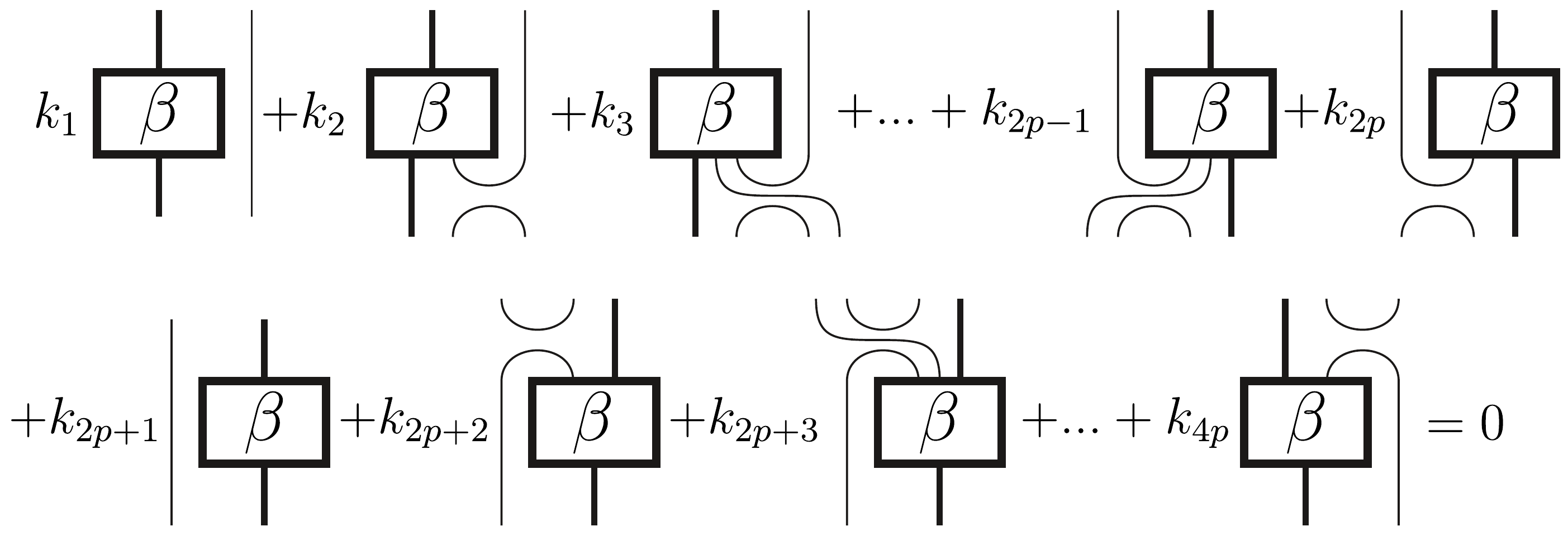}
 	\caption{Relation \ref{eq:16}}
 \end{figure}
 We also have the partial traces given by:
 \begin{figure}[H]
 	\centering
 	\includegraphics[width=0.6\linewidth]{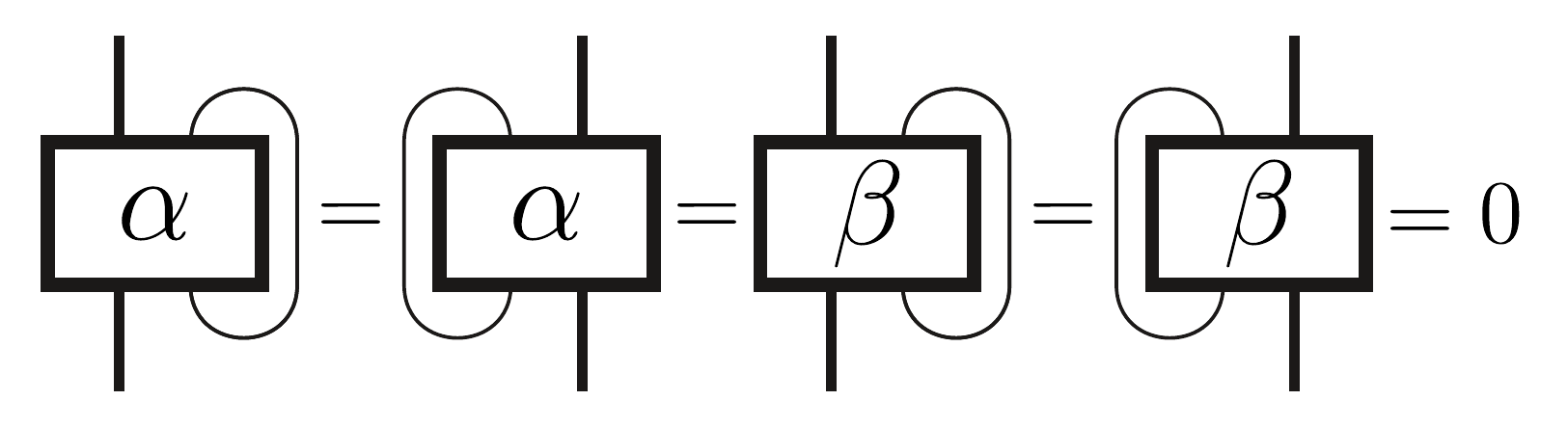}
 	\caption{The partial trace of $\alpha$ and $\beta$.}
 \end{figure}
 \begin{figure}[H]
 	\centering
 	\includegraphics[width=0.6\linewidth]{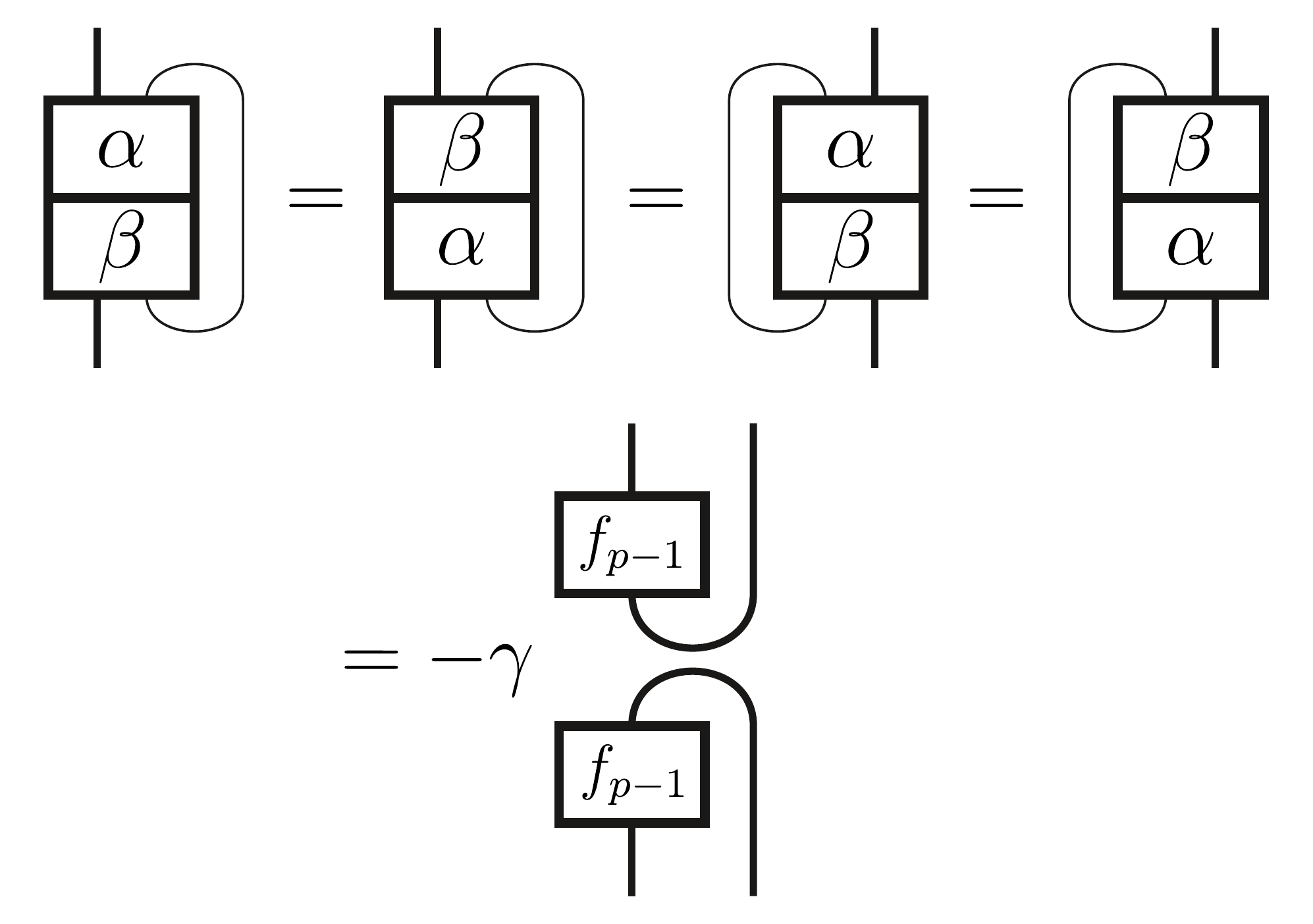}
 	\caption{The partial trace of $\alpha\beta$ and $\beta\alpha$.} \label{fig:pt}
 \end{figure}
\end{thm}
\noindent Relations between $\alpha$ and $\beta$ and the Temperley-Lieb generators $\mathbf{e}_{i}$ follow immediately from relations \ref{eq:8}-\ref{eq:12}:
\begin{align}
\mathbf{e}_{i}\alpha_{j}=\alpha_{j} \mathbf{e}_{i}&=\mathbf{e}_{i}\beta_{j}=\beta_{j} \mathbf{e}_{i}=0, \:\:\: 0\leq i-j\leq 2p-3 \label{eq:17}\\
\mathbf{e}_{i}\alpha_{i+1}&=\mathbf{e}_{i}\mathbf{e}_{i+1}...\mathbf{e}_{i+2p-3}\mathbf{e}_{i+2p-2}\alpha_{i} \label{eq:18}\\
\alpha_{i+1}\mathbf{e}_{i}&=\alpha_{i}\mathbf{e}_{i+2p-2}\mathbf{e}_{i+2p-3}...\mathbf{e}_{i+1}\mathbf{e}_{i} \label{eq:19}\\
\mathbf{e}_{i}\beta_{i+1}&=\mathbf{e}_{i}\mathbf{e}_{i+1}...\mathbf{e}_{i+2p-3}\mathbf{e}_{i+2p-2}\beta_{i} \label{eq:20}\\
\beta_{i+1}\mathbf{e}_{i}&=\beta_{i}\mathbf{e}_{i+2p-2}\mathbf{e}_{i+2p-3}...\mathbf{e}_{i+1}\mathbf{e}_{i} \label{eq:21}
\end{align}
Relation \ref{eq:17} was given in Section 4.7 of \cite{GST}, and relations \ref{eq:18}-\ref{eq:21} generalize their results in Appendix C for $p=2$ to all $p$.\\

We note that there are generalizations of relations \ref{eq:15} and \ref{eq:16} to higher $n>2p$. These generalizations can be stated as any sum over all diagrams with $n$ strings along the top and bottom and containing a single $\alpha$ (respectively $\beta$) is linearly dependent. They further generalize to sums over all diagrams containing multiple $\alpha$ or $\beta$ for $n\geq 3p$. We omitted to include these generalizations in Theorem \ref{thm} as we expect they can be constructed from the relations already given.

\section{Proofs of Relations.}
The aim for the remainder of this paper is to give proofs of the relations in Theorem \ref{thm}. Relation $1$ follows immediately from the action of $\alpha$ and $\beta$ on the weight spaces.
\subsection{Proof that $End(X^{\otimes(2p-1)})\simeq TL_{2p-1}(q+q^{-1})\oplus\mathbb{K}\alpha\oplus\mathbb{K}\beta\oplus\mathbb{K}\alpha\beta$.\label{subsect 2}}
From Proposition \ref{prop1} we have that $\dim\left( End(X^{\otimes(2p-1)})\right)=\dim\left( TL_{2p-1}(q+q^{-1})\right) +3$. From the actions of $\alpha$, $\beta$, and $\alpha\beta$ in terms of weight spaces, we know that they are not in the Temperley-Lieb algebra. Hence we only need to check that $\alpha,\beta\in End(X^{\otimes(2p-1)})$, which is straightforward.\\

Note that $Im(\alpha)$ and $Im(\beta)$ are both $p$ dimensional, with weights $\{q^{2p-1},q^{2p-3},...,q^{1-2p}\}$\\ $=\{-q^{p-1},-q^{p-3},...,-q^{1-p}\}$, and so $Im(\alpha)\simeq Im(\beta)\simeq\mathcal{X}^{-}_{p}$.

\subsection{Relations \ref{eq:2} and \ref{eq:3}, $\alpha\beta\alpha=\gamma\alpha$, $\beta\alpha\beta=\gamma\beta$.\label{subsect 3}}
Let $x\in X_{k,2p-1}$, and $e_{x},f_{x}\in\mathbb{K}$, such that $E^{k}x=e_{x}x_{0}$, $F^{2p-1-k}=f_{x}x_{2p-1}$. Using equations \ref{eq:A4} and \ref{eq:A5}, it follows that:
\begin{align*}
F^{j}\alpha(x)=& e_{x}\frac{([k+j]!)([2p-k-1]!)}{([k]!)([2p-k-j-1]!)}E^{p-k-j-1}x_{2p-1}\\
E^{j}\beta(x)=& f_{x}\frac{([2p-k+j-1]!)([k]!)}{([2p-k-1]!)([k-j]!)}F^{k-p-j}x_{0}
\end{align*}
We can use these to apply $\alpha$ and $\beta$ repeatedly, giving:
\begin{align*}
\beta(\alpha(x))=& e_{x}\frac{([2p-k-1]!)}{([k]!)[p]}F^{k}x_{0}, \:\: 0\leq k\leq p-1\\
\alpha(\beta(x))=& f_{x}\frac{([k]!)}{([2p-k-1]!)[p]}E^{2p-k-1}x_{2p-1}, \:\: p\leq k\leq 2p-1\\
\alpha(\beta(\alpha(x)))=& e_{x}\frac{([2p-k-1]!)([k+p]!)}{([k]!)([p-k-1]!)[p]^{2}}E^{p-k-1}x_{2p-1}, \:\: 0\leq k\leq p-1\\
\beta(\alpha(\beta(x)))=& f_{x}\frac{([3p-k-1]!)([k]!)}{([k-p]!)([2p-k-1]!)[p]^{2}}F^{k-p}x_{0}, \:\: p\leq k\leq 2p-1
\end{align*}
Simplifying the coefficients in terms of our choice of $p$, we have:
\begin{align*}
\frac{([2p-k-1]!)([k+p]!)}{([k]!)([p-k-1]!)[p]^{2}}=&(-1)^{p-1}([p-1]!)^{2}, \:\: 0\leq k\leq p-1\\
\frac{([3p-k-1]!)([k]!)}{([k-p]!)([2p-k-1]!)[p]^{2}}=&(-1)^{p-1}([p-1]!)^{2}, \:\: p\leq k\leq 2p-1
\end{align*}
Let $\gamma:=(-1)^{p-1}([p-1]!)^2$, then it follows that:
\begin{align*}
\alpha(\beta(\alpha(x)))=&\gamma\alpha(x)\\
\beta(\alpha(\beta(x)))=&\gamma\beta(x)
\end{align*}

Given that $Im(\alpha)\simeq Im(\beta)\simeq \mathcal{X}^{-}_{p}$, it follows that $\gamma^{-1}\alpha\beta$ and $\gamma^{-1}\beta\alpha$ are the projections onto the two copies of $\mathcal{X}^{-}_{p}$ in $X^{\otimes (2p-1)}$.

\subsection{Relation \ref{eq:4}, $\alpha_{i}\alpha_{j}=\alpha_{j}\alpha_{i}=\beta_{i}\beta_{j}=\beta_{j}\beta_{i}=0, \:\: \lvert i-j\rvert< p$.\label{subsect 4}}
This follows from considering $\alpha_{1}\alpha_{1+k}$, $\alpha_{1+k}\alpha_{1}$, $\beta_{1}\beta_{1+k}$, and $\beta_{1+k}\beta_{1}$ acting on $X^{\otimes 2p-1-k}$. We give the proof for the case $\alpha_{1}\alpha_{1+k}$. The other cases follow similar arguments.\\

Let $x\in X_{j,k}$, $0\leq j\leq k$, with $E^{j}x=e_{x}x_{0,k}$, $y\in X_{l,2p-1-k}$, $0\leq y\leq 2p-1-k$, with $E^{l}Y=e_{y}x_{0,2p-1-k}$, and $z\in X_{m,k}$, $0\leq m\leq k$, with $E^{m}z=e_{z}x_{0,k}$. Consider $\alpha_{1}(\alpha_{1+k}(x\otimes y\otimes z))$. By use of equation \ref{eq:A6}, we have:
\begin{align*}
\alpha_{1+k}(x\otimes y\otimes z)=&\sum\limits_{i=0}^{p-l-m-1}e_{y}e_{z}q^{lk}\lambda_{l,l+m}\lambda_{i,p-l-m-1}x\otimes(E^{i}x_{2p-1-k,2p-1-k})\otimes(K^{i}E^{p-l-m-1-i}x_{k,k})
\end{align*}
where $\lambda_{l,l+m},\lambda_{i,p-l-m-1}\in\mathbb{K}$. To act $\alpha_{1}$ on this, we need $E^{2p-1-k}$ acting on $x_{2p-1-k,2p-1-k}$ in the middle tensor term. However, as $E^{p}=0$, this will be zero if $2p-1-k\geq p$, and hence $\alpha_{1}\alpha_{1+k}=0$ if $k\leq p-1$.

\subsection{Relations \ref{eq:5} and \ref{eq:6}, $\alpha_{i}\alpha_{i+p}=\alpha_{i+p}\alpha_{i}$, $\beta_{i}\beta_{i+p}=\beta_{i+p}\beta_{i}$.}
We prove the case $\alpha_{1}\alpha_{1+p}=\alpha_{1+p}\alpha_{1}$ which generalizes to the above relation. The proof for\\ $\beta_{i}\beta_{i+p}=\beta_{i+p}\beta_{i}$ follows similarly. Using the same notation as Section \ref{subsect 4}, we have\\ $\alpha_{1}(x\otimes y\otimes z)=$
\begin{align*}
& \sum\limits_{i=0}^{p-j-l-1}\lambda_{j,j+l}e_{x}e_{y}q^{-j(p+1)}\lambda_{i,p-j-l-1}(E^{i}x_{p,p})\otimes(K^{i}E^{p-j-l-i-1}x_{p-1,p-1})\otimes z
\end{align*}
Then we have $E^{j+l+i+m}((K^{i}E^{p-j-l-i-1}x_{p-1,p-1})\otimes z)=$
\begin{align*}
& \lambda_{j+l+i,j+l+i+m}e_{z}q^{2i(p-j-l-i-1)+i(1-p)+p(j+l+i)}([p-1]!)x_{0,2p-1}
\end{align*}
As $E^{p}=0$, we only need to consider the terms where $j+l+i+m\leq p-1$, which gives\\ $i\leq p-j-l-m-1$. Hence we have $\alpha_{1+p}(\alpha_{1}(x\otimes y\otimes z))=$
\begin{align*}
& \sum\limits_{i=0}^{p-j-l-m-1}\lambda_{j,j+l}e_{x}e_{y}q^{-j(p+1)}\lambda_{i,p-j-l-1}\lambda_{j+l+i,j+l+i+m}e_{z}q^{2i(p-j-l-i-1)+i(1-p)+p(j+l+i)}\times\\
&\times([p-1]!)(E^{i}x_{p,p})\otimes(E^{p-1-j-l-i-m}x_{2p-1,2p-1})\\
=&\sum\limits_{i=0}^{p-j-l-m-1}\bigg(\sum\limits_{n=0}^{p-1-j-l-i-m}e_{x}e_{y}e_{z}\lambda_{j,j+l}\lambda_{i,p-j-l-1}\lambda_{j+l+i,j+l+i+m}q^{(2ip-i-j-2ij-2il-2i^{2}+pl)}\times\\
&\times\lambda_{n,p-1-j-l-i-m}([p-1]!)(E^{i}x_{p,p})\otimes(E^{n}x_{p-1,p-1})\otimes(K^{n}E^{p-1-j-l-i-m-n}x_{p,p})\bigg)
\end{align*}
Next we have $\alpha_{1+p}(x\otimes y\otimes z)=$
\begin{align*}
& \sum\limits_{r=0}^{p-l-m-1}e_{y}e_{z}q^{lp}\lambda_{l,l+m}\lambda_{r,p-l-m-1}x\otimes(E^{r}x_{p-1,p-1})\otimes(K^{r}E^{p-l-m-r-1}x_{p,p})
\end{align*}
Then we have:
\begin{align*}
E^{j+p-r-1}(x\otimes(E^{r}x_{p-1,p-1}))=& e_{x}\lambda_{j,j+p-r-1}q^{j(p-1)}([p-1]!)x_{0,2p-1}
\end{align*}
Again as $E^{p}=0$, we only need consider the terms where $j+p-r-1\leq p-1$, which gives $r\geq j$. Hence we have $\alpha_{1}(\alpha_{1+p}(x\otimes y\otimes z))=$
\begin{align*}
& \sum\limits_{r=j}^{p-l-m-1}e_{y}e_{z}q^{lp}\lambda_{l,l+m}\lambda_{r,p-l-m-1}e_{x}\lambda_{j,j+p-r-1}q^{j(p-1)}([p-1]!)\times\\
&\times(E^{r-j}x_{2p-1,2p-1})\otimes(K^{r}E^{p-l-m-r-1}x_{p,p})\\
=&\sum\limits_{r=j}^{p-l-m-1}\bigg(\sum\limits_{s=0}^{r-j}e_{x}e_{y}e_{z}\lambda_{l,l+m}\lambda_{r,p-l-m-1}\lambda_{j,j+p-r-1}q^{lp+j(p-1)}([p-1]!)\times\\
&\times\lambda_{s,r-j}(E^{s}x_{p,p})\otimes(K^{s}E^{r-j-s}x_{p-1,p-1})\otimes(K^{r}E^{p-l-m-r-1}x_{p,p})\bigg)
\end{align*}
Let $t=r-j$, then this becomes:
\begin{align*}
&\sum\limits_{t=0}^{p-l-m-j-1}\bigg(\sum\limits_{s=0}^{t}e_{x}e_{y}e_{z}\lambda_{l,l+m}\lambda_{t+j,p-l-m-1}\lambda_{j,p-t-1}q^{lp+j(p-1)}([p-1]!)\times\\
&\times\lambda_{s,t}(E^{s}x_{p,p})\otimes(K^{s}E^{t-s}x_{p-1,p-1})\otimes(K^{t+j}E^{p-l-m-j-t-1}x_{p,p})\bigg)
\end{align*}
Using the summation identity $\sum\limits_{u=0}^{w}\sum\limits_{v=0}^{u}x_{u,v}=\sum\limits_{v=0}^{w}\sum\limits_{u=v}^{w}x_{u,v}$, this becomes:
\begin{align*}
&\sum\limits_{s=0}^{p-l-m-j-1}\bigg(\sum\limits_{t=s}^{p-l-m-j-1}e_{x}e_{y}e_{z}\lambda_{l,l+m}\lambda_{t+j,p-l-m-1}\lambda_{j,p-t-1}q^{lp+j(p-1)}([p-1]!)\times\\
&\times\lambda_{s,t}(E^{s}x_{p,p})\otimes(K^{s}E^{t-s}x_{p-1,p-1})\otimes(K^{t+j}E^{p-l-m-j-t-1}x_{p,p})\bigg)
\end{align*}
Let $n=t-s$, then we have:
\begin{align*}
&\sum\limits_{s=0}^{p-l-m-j-1}\bigg(\sum\limits_{n=0}^{p-l-m-j-1-s}e_{x}e_{y}e_{z}\lambda_{l,l+m}\lambda_{n+s+j,p-l-m-1}\lambda_{j,p-n-s-1}q^{lp+j(p-1)}([p-1]!)\times\\
&\times\lambda_{s,n+s}(E^{s}x_{p,p})\otimes(K^{s}E^{n}x_{p-1,p-1})\otimes(K^{n+s+j}E^{p-l-m-j-n-s-1}x_{p,p})\bigg)
\end{align*}
Letting $s=i$ we have:
\begin{align*}
=&\sum\limits_{i=0}^{p-l-m-j-1}\bigg(\sum\limits_{n=0}^{p-l-m-j-1-i}e_{x}e_{y}e_{z}\lambda_{l,l+m}\lambda_{n+i+j,p-l-m-1}\lambda_{j,p-n-i-1}([p-1]!)\times\\
&\times q^{(lp-i-2i^{2}-3j-4ij-2j^{2}-2il-2jl-2im-2jm-2jn+2jp)}\lambda_{i,n+i}\times\\
&\times(E^{i}x_{p,p})\otimes(E^{n}x_{p-1,p-1})\otimes(K^{n}E^{p-l-m-j-n-i-1}x_{p,p})\bigg)
\end{align*} 
This is now the same summation as $\alpha_{1+p}\alpha_{1}$, hence we want to show that the coefficients are equal for both. We then want to show:
\begin{align*}
& q^{(2ip-j-2ij-2il-2i^{2}+pl)}\lambda_{j,j+l}\lambda_{i,p-j-l-1}\lambda_{j+l+i,j+l+i+m}\lambda_{n,p-1-j-l-i-m}\\
=& q^{(jp-2il-2im-4ij-2i^{2}-i-2jl-2jm-2j^{2}-2jn-2j)}\lambda_{l,l+m}\lambda_{n+i+j,p-l-m-1}\lambda_{j,p-n-i-1}\lambda_{i,n+i}
\end{align*}
for $0\leq i\leq p-l-m-j-1$ and $0\leq n\leq p-l-m-j-1-i$. For this we need to use equation \ref{eq:A8}.
The coefficients then simplify to give:
\begin{align*}
&\frac{([j+l]!)([p-j-l-1]!)([j+l+i+m]!)([p-1-j-l-i-m]!)}{([p-j-l-i-1]!)([j+l+i]!)}\\
=& q^{(-2ip)}\frac{([l+m]!)([p-l-m-1]!)([p-n-i-1]!)([n+i]!)}{([n+i+j]!)([p-n-i-1-j]!)}
\end{align*}
As $[p-x]=[x]$, we have $([x]!)([p-1-x]!)=([p-1]!)$. Therefore it reduces to:
\begin{align*}
([p-1]!)=q^{-2ip}([p-1]!)
\end{align*}
Hence the coefficients are equal, and so we have $\alpha_{1}\alpha_{1+p}=\alpha_{1+p}\alpha_{1}$.

\subsection{Relation \ref{eq:7}, $\alpha\beta+\beta\alpha=\gamma \mathfrak{f}_{2p-1}$. \label{subsect 6}}
From Section \ref{subsect 3}, we have:
\begin{align*}
\beta(\alpha(\rho_{i_{1},...,i_{k},2p-1}))=& q^{\big(k(2p-1)-\frac{1}{2}(k^{2}-k)-(\sum\limits_{j=1}^{k}i_{j})\big)}\frac{([2p-k-1]!)}{[p]}F^{k}x_{0}, \:\: 0\leq k\leq p-1\\
\alpha(\beta(\rho_{i_{1},...,i_{k},2p-1}))=& q^{\big(k(2p-1)-\frac{1}{2}(k^{2}-k)-(\sum\limits_{j=1}^{k}i_{j})\big)}\frac{([k]!)}{[p]}E^{2p-k-1}x_{2p-1}, \:\: p\leq k\leq 2p-1
\end{align*}
The Jones-Wenzl projection, $\mathfrak{f}_{2p-1}$, is given by:
\begin{align*}
\mathfrak{f}_{2p-1}(\rho_{i_{1},...,i_{k},2p-1})=& q^{\big(k(2p-1)-\frac{1}{2}(k^{2}-k)-(\sum\limits_{j=1}^{k}i_{j})\big)}\frac{([2p-k-1]!)}{([2p-1]!)}F^{k}x_{0}\\
\gamma \mathfrak{f}_{2p-1}(\rho_{i_{1},...,i_{k},2p-1})=& q^{\big(k(2p-1)-\frac{1}{2}(k^{2}-k)-(\sum\limits_{j=1}^{k}i_{j})\big)}\frac{([2p-k-1]!)}{[p]}F^{k}x_{0}
\end{align*}
Hence $\beta(\alpha(\rho_{i_{1},...,i_{k},2p-1}))=\gamma \mathfrak{f}_{2p-1}(\rho_{i_{1},...,i_{k},2p-1})$, for $0\leq k\leq p-1$. Note, that although we should consider this equal to zero for $k\geq p$, here we are only using $\frac{F^{k}}{[p]}$ to represent the relevant element of $X^{\otimes(2p-1)}$, which is non-zero. From equations \ref{eq:A11} and \ref{eq:A12}, we have:
\begin{align*}
([z-k]!)F^{k}x_{0,z}=&([k]!)E^{z-k}x_{z,z}
\end{align*}
It follows that $\alpha(\beta(\rho_{i_{1},...,i_{k},2p-1}))=\gamma \mathfrak{f}_{2p-1}(\rho_{i_{1},...,i_{k},2p-1})$, for $p\leq k\leq 2p-1$. Hence $\alpha\beta+\beta\alpha=\gamma \mathfrak{f}_{2p-1}$.

\subsection{Relation \ref{eq:8}, $\alpha\cap_{i}=\cup_{i}\alpha=\beta\cap_{i}=\cup_{i}\beta=0, \:\: 1\leq i\leq 2p-2$.}
From Proposition \ref{prop 3}, we have that the image of $\alpha$ and $\beta$ is $\mathcal{X}^{-}_{p}$. However, $\mathcal{X}^{-}_{p}$ is irreducible, and does not appear in the decomposition of $X^{\otimes n}$ until $n=2p-1$. Capping or cupping $\alpha$ or $\beta$ gives a map between $\mathcal{X}^{-}_{p}$ and $X^{\otimes(2p-3)}$. Hence this map must be zero.

\subsection{Relations \ref{eq:9} and \ref{eq:10}, $\alpha_{i+1}\cap_{i}=\alpha_{i}\cap_{i+2p-2}$, $\beta_{i+1}\cap_{i}=\beta_{i}\cap_{i+2p-2}$.}

Given $\cap(\nu)=q^{-1}\nu_{10}-\nu_{01}$, we can write relation \ref{eq:9} explicitly as:
\begin{align*}
& q^{-1}\alpha(\rho_{i_{1},...,i_{n},2p-2}\otimes\nu_{1})\otimes\nu_{0}-\alpha(\rho_{i_{1},...,i_{n},2p-2}\otimes\nu_{0})\otimes\nu_{1}\\
=& q^{-1}\nu_{1}\otimes\alpha(\nu_{0}\otimes\rho_{i_{1},...,i_{n},2p-2})-\nu_{0}\otimes\alpha(\nu_{1}\otimes\rho_{i_{1},...,i_{n},2p-2})
\end{align*}
Using equation \ref{eq:A6}, this simplifies to become:
\begin{align*}
& q^{n-1}\lambda_{n,n+1}(E^{p-n-2}x_{2p-1,2p-1})\otimes\nu_{0}-q^{n}(E^{p-n-1}x_{2p-1,2p-1})\otimes\nu_{1}\\
=& q^{-1}\nu_{1}\otimes(E^{p-n-1}x_{2p-1,2p-1})-q^{2p-2}\lambda_{1,n+1}\nu_{0}\otimes(E^{p-n-2}x_{2p-1,2p-1})
\end{align*}
Using equation \ref{eq:A13} and \ref{eq:A14}, both sides can be shown to equal $q^{-1}E^{p-n-1}x_{2p,2p}$. The proof of relation \ref{eq:10} is similar, with both sides reducing to $q^{-1}F^{p-n-1}x_{0,2p}$.

\subsection{Relations \ref{eq:11}, \ref{eq:12}, \ref{eq:13}, and \ref{eq:14}, $\cup_{i}\alpha_{i+1}=\cup_{i+2p-2}\alpha_{i}$, $\cup_{i}\beta_{i+1}=\cup_{i+2p-2}\beta_{i}$, $R_{4p-2}(\alpha)=\alpha$, $R_{4p-2}(\beta)=\beta$.}
These can be proven diagrammatically from relations \ref{eq:9} and \ref{eq:10}. We demonstrate the proof for $\alpha$, with $\beta$ following similarly.
\begin{figure}[H]
	\centering
	\includegraphics[width=0.7\linewidth]{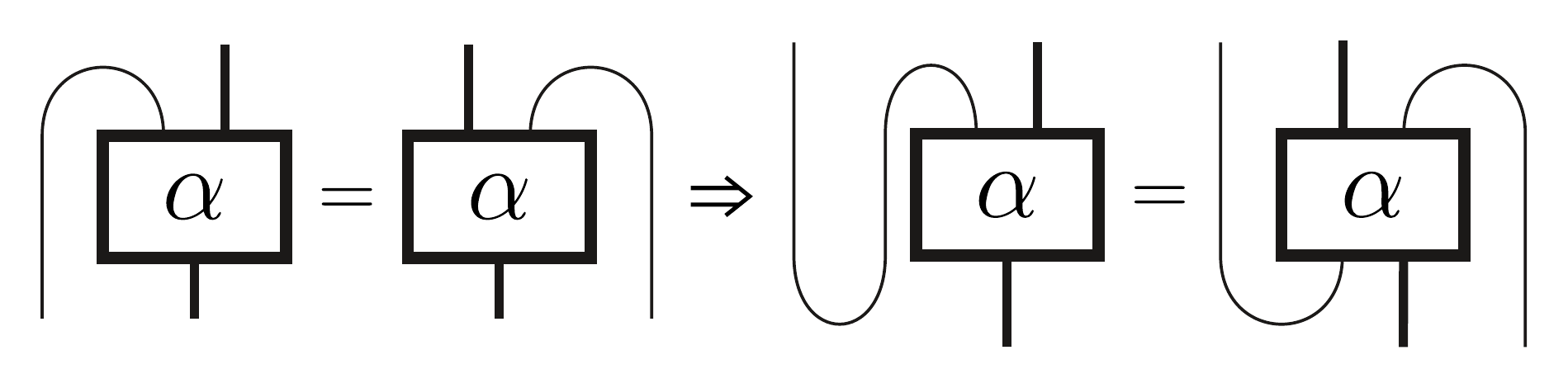}
	\includegraphics[width=0.35\linewidth]{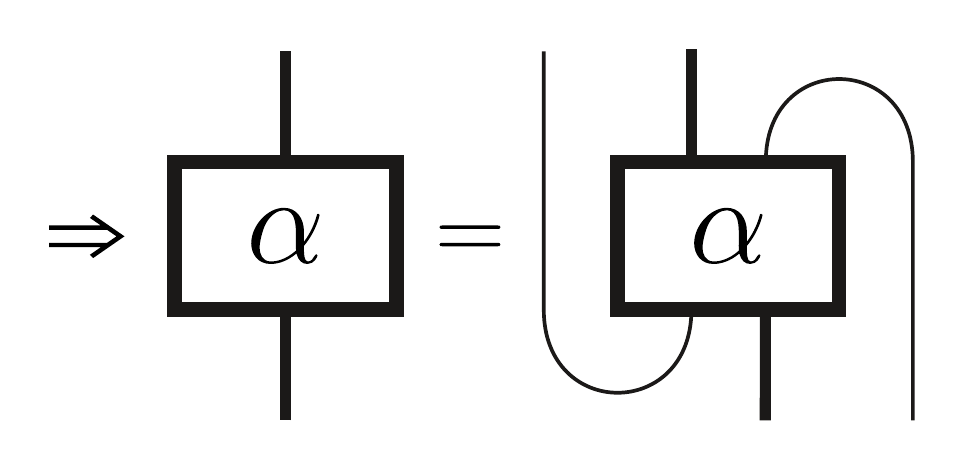}
\end{figure}
Hence $\alpha$ is rotation invariant. Continuing with this, we have:
\begin{figure}[H]
	\centering
	\includegraphics[width=0.7\linewidth]{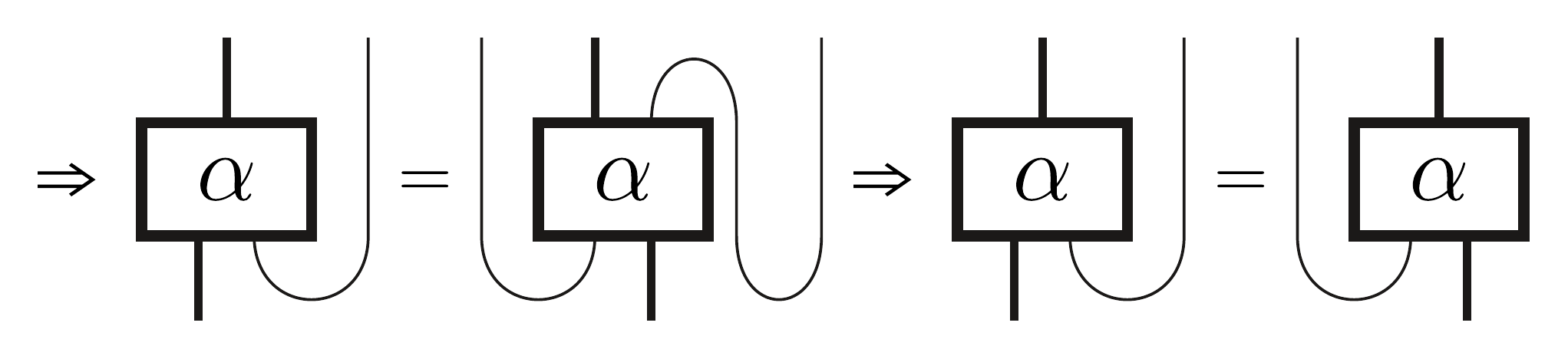}
\end{figure}
and so we have the cupping relation.

\subsection{Partial Traces}
The partial trace of $\alpha$ and $\beta$ can be derived easily from the cupping and capping relations (relations \ref{eq:9} - \ref{eq:12}), and shown to be zero. We demonstrate the case for alpha:
\begin{figure}[H]
	\centering
	\includegraphics[width=0.7\linewidth]{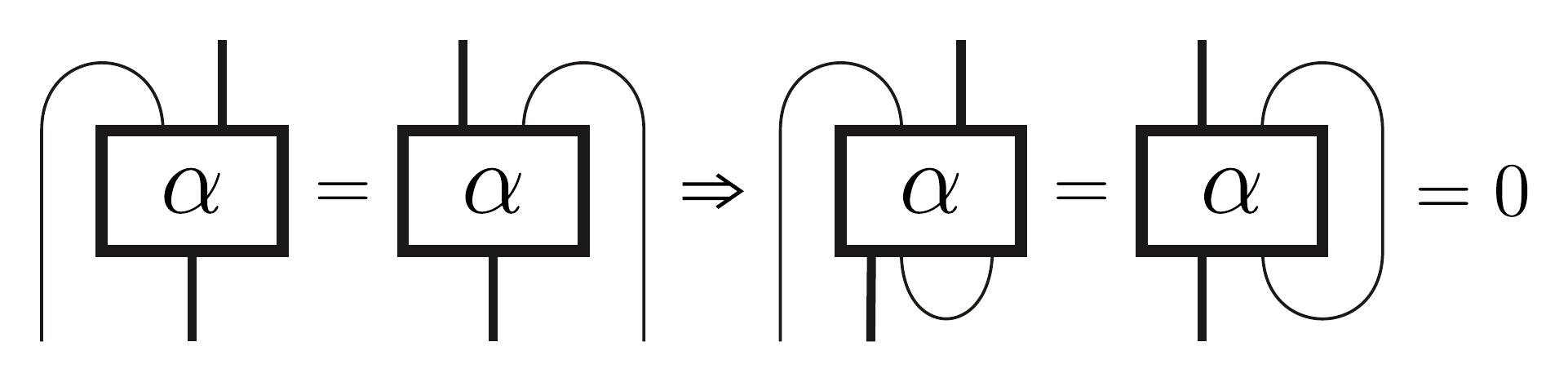}
\end{figure}
The partial trace of $\alpha\beta$ and $\beta\alpha$ is given in Figure \ref{fig:pt}. Note that the left partial trace follows from the right partial trace and the cupping and capping relations. We give the proof for $\beta\alpha$, with the proof for $\alpha\beta$ following similarly.\\
\\
Consider $X^{\otimes 2p-2}\otimes\cap$. Its elements take the form $q^{-1}\rho_{i_{1},...,i_{n},2p-2}\otimes\nu_{10}-\rho_{i_{1},...,i_{n},2p-2}\otimes\nu_{01}$. We want to apply $(\beta\alpha\otimes 1)$ to this. From Section \ref{subsect 3}, we had that given $x\in X_{k,2p-1}$, $0\leq k\leq p-1$,
\begin{align*}
\beta(\alpha(x))&=e_{x}\frac{([2p-k-1]!)}{([k]!)[p]}F^{k}x_{0,2p-1}
\end{align*}
We then have:
\begin{align*}
\beta(\alpha(\rho_{i_{1},...,i_{n},2p-2}\otimes\nu_{1}))=& q^{\big(n(2p-1)-\frac{1}{2}(n^{2}-n)-(\sum\limits_{j=1}^{n}i_{j})\big)}\lambda_{n,n+1}\frac{([2p-n-2]!)}{[n+1][p]}\times\\
&\times\bigg((F^{n+1}x_{0,2p-2})\otimes\nu_{0}+q^{n-2p+2}[n+1](F^{n}x_{0,2p-2})\otimes\nu_{1}\bigg)\\
\beta(\alpha(\rho_{i_{1},...,i_{n},2p-2}\otimes\nu_{0}))=& q^{\big(n(2p-1)-\frac{1}{2}(n^{2}-n)-(\sum\limits_{j=1}^{n}i_{j})\big)}\frac{([2p-n-1]!)}{[p]}\times\\
&\times\bigg((F^{n}x_{0,2p-2})\otimes\nu_{0}+q^{n-2p+1}[n](F^{n-1}x_{0,2p-2})\otimes\nu_{1}\bigg)
\end{align*}
Note that as $\rho_{i_{1},...,i_{n},2p-2}\otimes\nu_{1}\in X_{n+1,2p-1}$, and $\alpha$ is zero on $X_{n,2p-1}$ for $n\geq p$, we will have to treat the case $n=p-1$ separately. For $0\leq n\leq p-2$, we have:\\ 

$(\beta\alpha\otimes 1)(q^{-1}\rho_{i_{1},...,i_{n},2p-2}\otimes\nu_{10}-\rho_{i_{1},...,i_{n},2p-2}\otimes\nu_{01})$
\begin{align*}
=& q^{\big(n(2p-1)-\frac{1}{2}(n^{2}-n)-1-(\sum\limits_{j=1}^{n}i_{j})\big)}\lambda_{n,n+1}\frac{([2p-n-2]!)}{[n+1][p]}\times\\
&\times\bigg((F^{n+1}x_{0,2p-2})\otimes\nu_{00}+q^{n-2p+2}[n+1](F^{n}x_{0,2p-2})\otimes\nu_{10}\bigg)\\
&-q^{\big(n(2p-1)-\frac{1}{2}(n^{2}-n)-(\sum\limits_{j=1}^{n}i_{j})\big)}\frac{([2p-n-1]!)}{[p]}\bigg((F^{n}x_{0,2p-2})\otimes\nu_{01}+q^{n-2p+1}[n](F^{n-1}x_{0,2p-2})\otimes\nu_{11}\bigg)
\end{align*}
Applying $1^{\otimes 2p-2}\otimes\cup$, this simplifies to give zero. For the case $n=p-1$, as $\rho_{i_{1},...,i_{p-1},2p-2}\otimes\nu_{1}\in X_{p,2p-1}$, we have\\

$(\beta\alpha\otimes 1)(q^{-1}\rho_{i_{1},...,i_{p-1},2p-2}\otimes\nu_{10}-\rho_{i_{1},...,i_{p-1},2p-2}\otimes\nu_{01})$
\begin{align*}
=&-q^{\big((p-1)(2p-1)-\frac{1}{2}(p-1)(p-2)-(\sum\limits_{j=1}^{p-1}i_{j})\big)}\frac{([p]!)}{[p]}\bigg((F^{p-1}x_{0,2p-2})\otimes\nu_{01}+q^{-p}[p-1](F^{p-2}x_{0,2p-2})\otimes\nu_{11}\bigg)
\end{align*}
Applying $1^{\otimes 2p-2}\otimes\cup$ to this, we get:
\begin{align*}
q^{\big((p-1)(2p-1)+1-\frac{1}{2}(p-1)(p-2)-(\sum\limits_{j=1}^{p-1}i_{j})\big)}([p-1]!)F^{p-1}x_{0,2p-2}
\end{align*}
We now want to give an explicit formula for the following diagram:
\begin{figure}[H]
	\centering
	\includegraphics[width=0.2\linewidth]{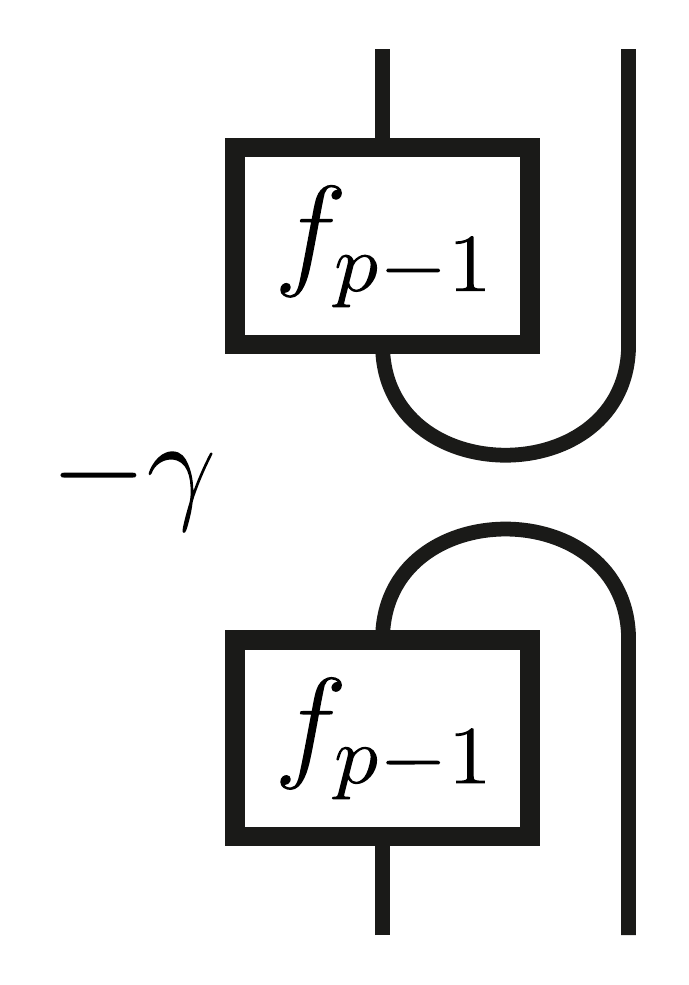}
\end{figure}
Given $\rho_{i_{1},...,i_{m},i_{m+1},...,i_{n},2p-2}$, we can rewrite it as $\rho_{i_{1},...,i_{m},p-1}\otimes\rho_{i_{m+1}+1-p,...,i_{n}+1-p,p-1}$. Applying $\mathfrak{f}_{p-1}\otimes 1^{\otimes p-1}$ to this we get:
\begin{align*}
& q^{\big(m(p-1)-\frac{1}{2}(m^{2}-m)-(\sum\limits_{j=1}^{m}i_{j})\big)}\frac{([p-1-m]!)}{([p-1]!)}(F^{m}x_{0,p-1})\otimes\rho_{i_{m+1}+1-p,...,i_{n}+1-p,p-1}\\
=& \sum\limits_{k_{1},...,k_{m}}q^{\big(mp-(\sum\limits_{j=1}^{m}i_{j})-(\sum\limits_{l=1}^{m}k_{l})\big)}\frac{([p-1-m]!)([m]!)}{([p-1]!)}\rho_{k_{1},...,k_{m},p-1}\otimes\rho_{i_{m+1}+1-p,...,i_{n}+1-p,p-1}
\end{align*}
Given $\cup(\nu_{10})=\nu$, $\cup(\nu_{01})=-q\nu$, $\cup(\nu_{00})=\cup(\nu_{11})=0$, applying cups repeatedly to this, we get zero if $n\neq p-1$ or if $\{k_{1},...,k_{m}\}\cap\{2p-1-i_{m+1},...,2p-1-i_{n}\}\neq\emptyset$. If $n=p-1$ and $\{k_{1},...,k_{m}\}\cap\{2p-1-i_{m+1},...,2p-1-i_{n}\}=\emptyset$, then we have:
\begin{align*}
&\sum\limits_{k_{1},...,k_{m}}(-1)^{p-1-m}q^{\big(mp+p-1-m-(\sum\limits_{j=1}^{m}i_{j})-(\sum\limits_{l=1}^{m}k_{l})\big)}\frac{([p-1-m]!)([m]!)}{([p-1]!)}\nu\\
=&\sum\limits_{k_{1},...,k_{m}}(-1)^{p-1}q^{\big(p-1-m-(\sum\limits_{j=1}^{m}i_{j})-(\sum\limits_{l=1}^{m}k_{l})\big)}\nu
\end{align*}
Note that for each choice of $i_{m+1},...,i_{n}$, there is a unique choice of $k_{1},...,k_{m}$ satisfying the above conditions, i.e. $\{k_{1},...,k_{m},2p-1-i_{m+1},...,2p-1-i_{n}\}=\{1,...,p-1\}$, and so we have that:
\begin{align*}
\sum\limits_{l=1}^{m}k_{l}=&\frac{1}{2}(p^{2}-p)-(2p-1)(n-m)+\sum\limits_{r=m+1}^{n}i_{r}
\end{align*}
Using this, we can simplify to get:
\begin{align*}
&\sum\limits_{k_{1},...,k_{m}}(-1)^{p-1}q^{\big(p-1-m-(\sum\limits_{j=1}^{m}i_{j})-(\sum\limits_{l=1}^{m}k_{l})\big)}\nu\\
=&(-1)^{p-1}q^{\big(p-1-m-(\sum\limits_{j=1}^{m}i_{j})-\frac{1}{2}(p^{2}-p)+(2p-1)(n-m)-(\sum\limits_{r=m+1}^{n}i_{r})\big)}\nu\\
=&(-1)^{p-1}q^{\big(-\frac{1}{2}(p^{2}-p)-(\sum\limits_{j=1}^{n}i_{j})\big)}\nu
\end{align*}
Given $\rho_{r_{1},...,r_{n},z}$, let $\tilde{r}_{1},...,\tilde{r}_{z-n}$ be the positions of the zeros. As $\cap(\nu)=q^{-1}\nu_{10}-\nu_{01}$, we have that the $z$-fold cap is given by:
\begin{align*}
\sum\limits_{n=0}^{z}\bigg(\sum\limits_{r_{1},...,r_{n}}(-1)^{z-n}q^{-n}\rho_{r_{1},...,r_{n},2z+1-\tilde{r}_{z-n},...,2z+1-\tilde{r}_{1},2z}\bigg)
\end{align*}
Taking $z=p-1$, this becomes:
\begin{align*}
\sum\limits_{n=0}^{p-1}\bigg(\sum\limits_{r_{1},...,r_{n}}(-1)^{p-1-n}q^{-n}\rho_{r_{1},...,r_{n},2p-1-\tilde{r}_{p-1-n},...,2p-1-\tilde{r}_{1},2p-2}\bigg)
\end{align*}
Applying $\mathfrak{f}_{p-1}\otimes 1^{\otimes p-1}$ to this we get:
\begin{align*}
&\sum\limits_{n=0}^{p-1}\bigg(\sum\limits_{r_{1},...,r_{n}}(-1)^{p-1-n}q^{\big(n(p-2)-\frac{1}{2}(n^{2}-n)-(\sum\limits_{j=1}^{n}r_{j})\big)}\times\\
&\times\frac{([p-1-n]!)}{([p-1]!)}(F^{n}x_{0,p-1})\otimes\rho_{p-\tilde{r}_{p-1-n},...,p-\tilde{r}_{1},p-1}\bigg)\\
=&\sum\limits_{n=0}^{p-1}\bigg(\sum\limits_{r_{1},...,r_{n}}\sum\limits_{s_{1},...,s_{n}}(-1)^{p-1-n}q^{\big(n(p-1)-(\sum\limits_{j=1}^{n}r_{j})-(\sum\limits_{k=1}^{n}s_{k})\big)}\times\\
&\times\frac{([p-1-n]!)([n]!)}{([p-1]!)}\rho_{s_{1},...,s_{n},p-1}\otimes\rho_{p-\tilde{r}_{p-1-n},...,p-\tilde{r}_{1},p-1}\bigg)
\end{align*}
As $\sum\limits_{j=1}^{n}r_{j}+\sum\limits_{l=1}^{p-1-n}\tilde{r}_{l}=\frac{1}{2}(p^{2}-p)$, this becomes:
\begin{align*}
&\sum\limits_{n=0}^{p-1}\bigg(\sum\limits_{\tilde{r}_{1},...,\tilde{r}_{p-1-n}}\sum\limits_{s_{1},...,s_{n}}(-1)^{p-1-n}q^{\big(n(p-1)-\frac{1}{2}(p^{2}-p)+(\sum\limits_{j=1}^{p-1-n}\tilde{r}_{j})-(\sum\limits_{k=1}^{n}s_{k})\big)}\times\\
&\times\frac{([p-1-n]!)([n]!)}{([p-1]!)}\rho_{s_{1},..,s_{n},2p-1-\tilde{r}_{p-1-n},...,2p-1-\tilde{r}_{1},2p-2}\bigg)\\
=&\sum\limits_{n=0}^{p-1}\bigg(\sum\limits_{t_{1},...,t_{p-1}}(-1)^{p-1-n}q^{\big(n(p-1)-\frac{1}{2}(p^{2}-p)+(2p-1)(p-1-n)-(\sum\limits_{l=1}^{p-1}t_{l})\big)}\rho_{t_{1},...,t_{p-1},2p-2}\bigg)\\
=&\sum\limits_{n=0}^{p-1}\bigg(\sum\limits_{t_{1},...,t_{p-1}}(-1)^{p-1}q^{\big(1-p-\frac{1}{2}(p^{2}-p)-(\sum\limits_{l=1}^{p-1}t_{l})\big)}\rho_{t_{1},...,t_{p-1},2p-2}\bigg)\\
=&\sum\limits_{t_{1},...,t_{p-1}}(-1)^{p-1}q^{\big(1-p-\frac{1}{2}(p^{2}-p)-(\sum\limits_{l=1}^{p-1}t_{l})\big)}\rho_{t_{1},...,t_{p-1},2p-2}\\
=& q^{(1-p^{2})}\frac{(-1)^{p-1}}{([p-1]!)}F^{p-1}x_{0,2p-2}
\end{align*}
Where we have taken $t_{1}:=s_{1},...,t_{n}:=s_{n}$, $t_{n+1}:=2p-1-\tilde{r}_{p-1-n},...,t_{p-1}:=2p-1-\tilde{r}_{1}$.
Combining this with the first part we get:
\begin{align*}
& q^{\big(1-p^{2}-\frac{1}{2}(p^{2}-p)-(\sum\limits_{j=1}^{n}i_{j})\big)} \frac{(-1)^{2p-2}}{([p-1]!)}F^{p-1}x_{0,2p-2}
\end{align*}
Multiplying by $-\gamma$, this becomes:
\begin{align*}
& (-1)^{p}q^{\big(1-p^{2}-\frac{1}{2}(p^{2}-p)-(\sum\limits_{j=1}^{n}i_{j})\big)}([p-1]!)F^{p-1}x_{0,2p-2}
\end{align*}
which is equal to the partial trace of $\beta\alpha$.\\
This map is also the second (non-identity) endomorphism on $\mathcal{P}^{+}_{1}$. Details of this endomorphism for all $\mathcal{P}^{\pm}_{i}$, as well as homomorphisms between indecomposable modules, will be discussed in a future paper \cite{Me2}.

\subsection{Relations \ref{eq:15} and \ref{eq:16}, $\sum\limits_{i=0}^{4p-1} k_{i}R^{i}_{4p}(\alpha\otimes 1)=0$, $\sum\limits_{i=0}^{4p-1} k_{i}R^{i}_{4p}(\beta\otimes 1)=0$.}
Let $R_{n}$ denote the clockwise rotation tangle acting on $n$ points.
\begin{figure}[H]
	\centering
	\includegraphics[width=0.2\linewidth]{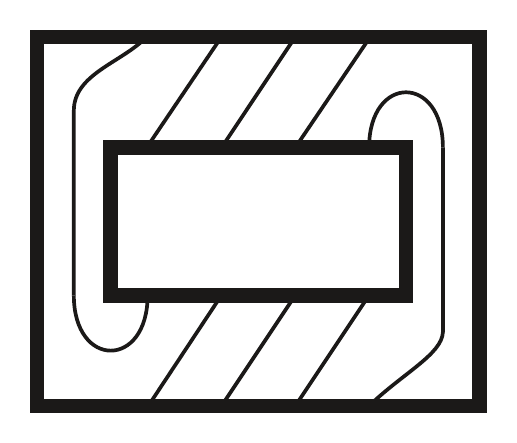}
	\caption{The rotation tangle $R_{8}$.}
\end{figure}
We want to show the relations $\sum\limits_{i=0}^{4p-1} k_{i}R^{i}_{4p}(\alpha\otimes 1)=0$, $\sum\limits_{i=0}^{4p-1} k_{i}R^{i}_{4p}(\beta\otimes 1)=0$. These in turn allow the proof of a large number of other relations, and simplification of diagrams containing $\alpha$ and $\beta$. The proof consists of two steps; first we show that the diagrams $\{R_{4p}^{i}\alpha\}$, $\{R^{i}_{4p}\beta\}$ are linearly dependent, so that the coefficients can be non-zero. We then give a general solution for the coefficients.\\

We demonstrate the proof for $\alpha$, with the proof for $\beta$ following similarly. Diagrammatically, $\sum\limits_{i=0}^{4p-1} k_{i}R^{i}_{4p}(\alpha\otimes 1)=0$ is given by Figure \ref{fig:rot}. From this, we see that there are $4p$ different diagrams, each of which acts on weight spaces by $X_{k,2p}\mapsto X_{k+p,2p}$. We want to show that the total number of maps acting as $X_{k,2p}\mapsto X_{k+p,2p}$ is less than $4p$. A list of module maps was given in Section \ref{Uq}, and from this we see that all module maps act as $X_{k,2p}\mapsto X_{k,2p}$, except for maps between $\mathcal{P}^{+}_{s}$ and $\mathcal{P}^{-}_{p-s}$ (As $\mathcal{X}^{-}_{p}$ does not appear in the decomposition of $X^{\otimes 2p}$). For example, for $p=3$, the maps not preserving weight spaces are:
\begin{figure}[H]
	\centering
	\includegraphics[width=0.6\linewidth]{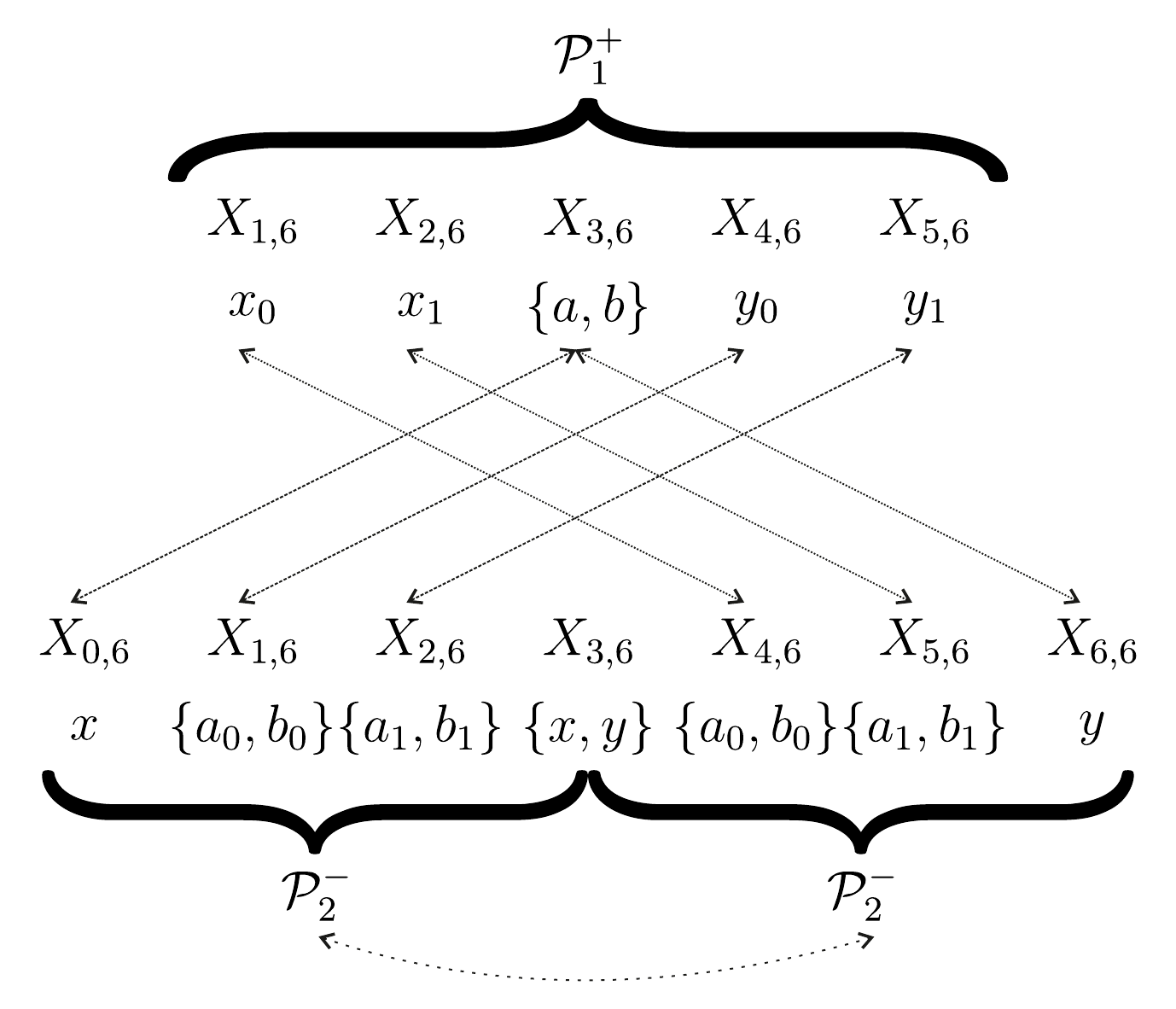}
\end{figure}
Hence there are two maps acting as $X_{k,2p}\mapsto X_{k+p,2p}$, and two acting as $X_{k,2p}\mapsto X_{k-p,2p}$. Denote the multiplicity of $\mathcal{P}^{+}_{1}$ in $X^{\otimes 2p}$ by $M(\mathcal{P}^{+}_{1})$. From this, we see that the total number of maps acting as $X_{k,2p}\mapsto X_{k+p,2p}$ is $2M(\mathcal{P}^{+}_{1})+2$. By considering module decompositions, we find that $M(\mathcal{P}^{+}_{1})=2p-2$. Hence the diagrams $\{R_{4p}^{i}\alpha\}$ are linearly dependent.\\

Consider applying $\cap_{2p-1}$ to $\sum\limits_{i=0}^{4p-1} k_{i}R^{i}_{4p}(\alpha\otimes 1)=0$. As capping off $\alpha$ gives zero, this reduces to $(k_{1}+\delta k_{4p}+k_{4p-1})\alpha_{1}\cap_{2p-1}=0$, so $k_{1}+\delta k_{4p}+k_{4p-1}=0$. Repeating this at every position, we get the general condition $k_{i-1}+\delta k_{i}+k_{i+1}=0$. Using this, we can rewrite the coefficients as:
\begin{align*}
k_{i}=& (-1)^{i}[i-2]k_{1}+(-1)^{i}[i-1]k_{2},& k_{i+p}=&(-1)^{p+1}k_{i}
\end{align*}
W denote $P_{\alpha}:=\sum\limits_{i=0}^{4p-1} k_{i}R^{i}_{4p}(\alpha\otimes 1)$, $P_{\beta}:=\sum\limits_{i=0}^{4p-1} k_{i}R^{i}_{4p}(\beta\otimes 1)$. These generalize a large number of relations. For example, consider $\beta_{2}P_{\alpha}\beta_{2}$, with $k_{1}=1$, $k_{2}=0$. This gives:
\begin{figure}[H]
	\centering
	\includegraphics[width=0.25\linewidth]{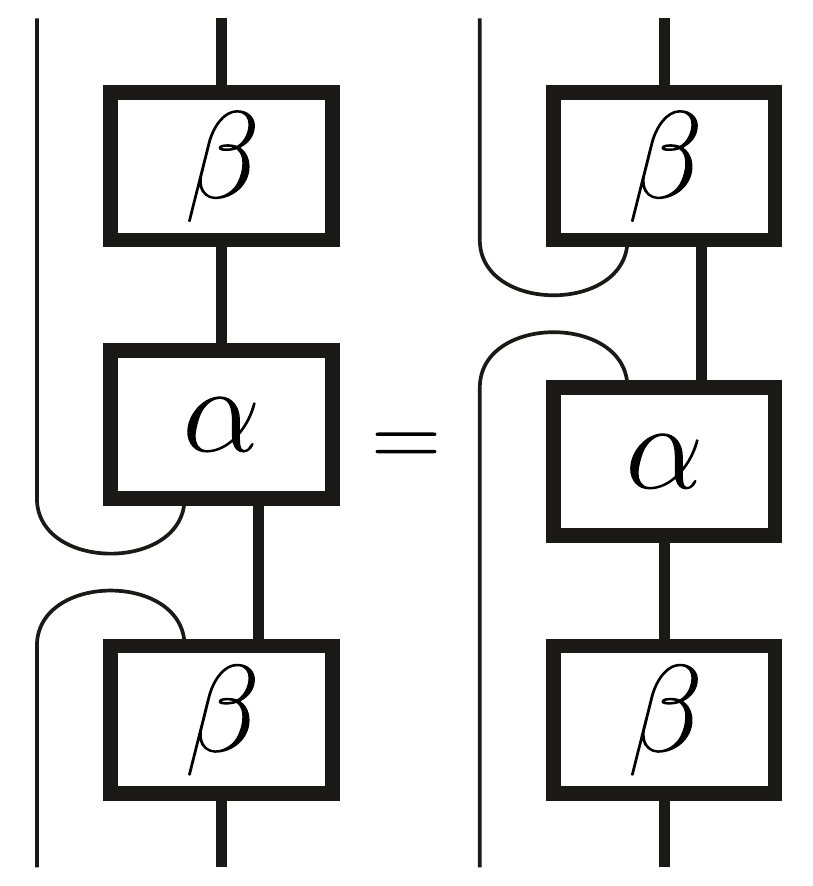}
\end{figure}
Considering $\alpha_{1}\beta_{1}P_{\alpha_{1}}$ with $k_{1}=1$, $k_{2}=0$, gives $k_{i}=(-1)^{i}[i-2]$ and:
\begin{figure}[H]
	\centering
	\includegraphics[width=0.7\linewidth]{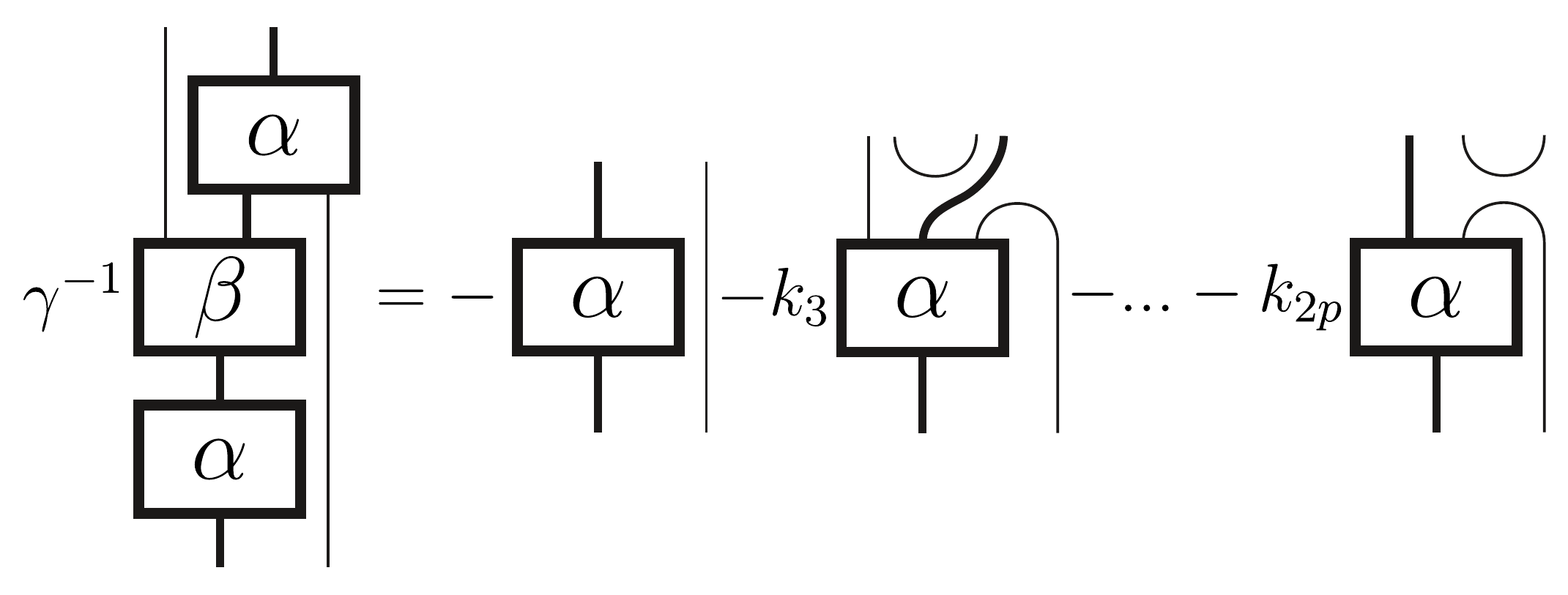}
\end{figure}
whereas $\alpha_{1}\beta_{1}P_{\alpha_{1}}$ with $k_{1}=0$, $k_{2}=1$, gives $k_{i}=(-1)^{i}[i-1]$ and:
\begin{figure}[H]
	\centering
	\includegraphics[width=0.7\linewidth]{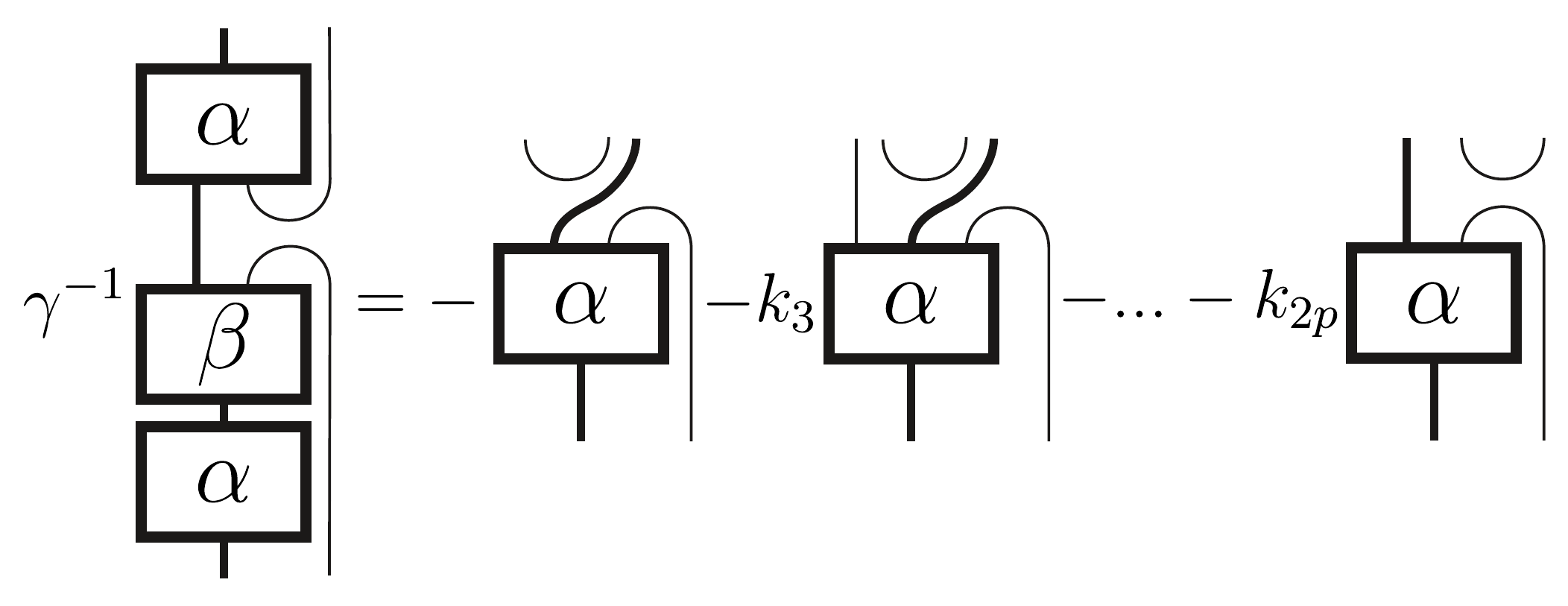}
\end{figure}
$P_{\alpha}$ and $P_{\beta}$ can also be used to reduce compositions of $\alpha$ and $\beta$. For example, consider $\beta_{3}(P_{\alpha_{1}}\otimes 1)$ with $k_{1}=1$, $k_{2p-1}=0$, which gives $k_{2}=-\delta^{-1}[3]$, $k_{i}=(-1)^{i+1}\delta^{-1}[3][i-1]+(-1)^{i}[i-2]$ and:
\begin{figure}[H]
	\centering
	\includegraphics[width=1\linewidth]{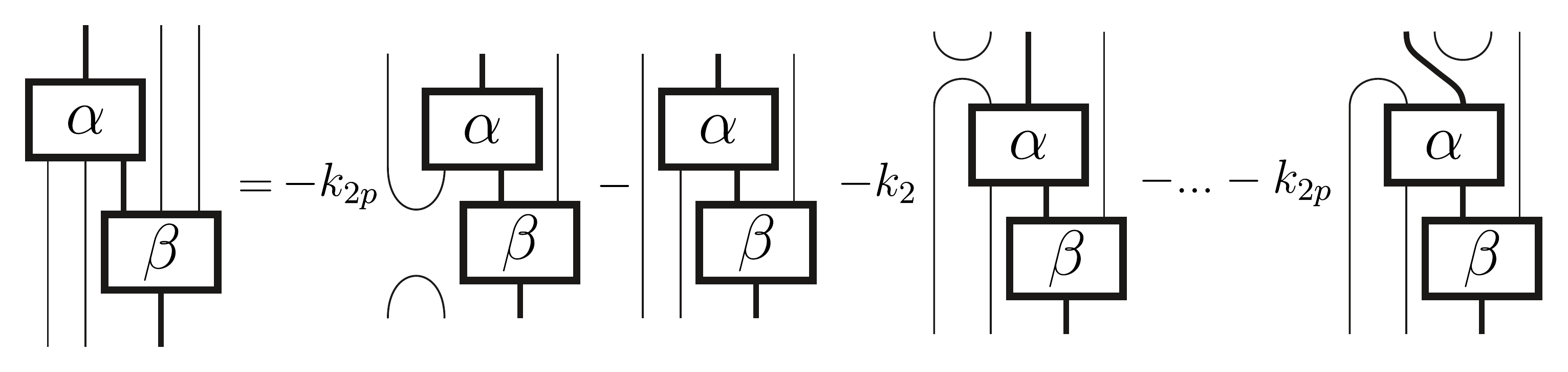}
\end{figure}
Note that this can not be used to reduce diagrams of the form $\alpha_{1}\beta_{p+1}$, $\beta_{1}\alpha_{p+1}$, etc. These diagrams describe maps between copies of $\mathcal{X}^{+}_{p}$ in $X^{\otimes(3p-1)}$.\\

Relations \ref{eq:15} and \ref{eq:16} can be taken as stating that at most $4p-2$ of the diagrams appearing in the relations are linearly independent, depending on the choice of the remaining two diagrams. Given that the dimension $D_{2p}=C_{2p}+12p-6$, we can use this to state the following:
\begin{prop}
$End_{\bar{U}_{q}(\mathfrak{sl}_{2})}(X^{\otimes 2p})$ has basis
\begin{align*}
\{&TL_{2p},\alpha_{1},\alpha_{2},\alpha_{2}e_{1},\alpha_{2}e_{1}e_{2},...,\alpha_{2}e_{1}e_{2}...e_{2p-2},e_{1}\alpha_{2},e_{2}e_{1}\alpha_{2},...,e_{2p-2}e_{2p-3}...e_{1}\alpha_{2},\\
&\beta_{1},\beta_{2},\beta_{2}e_{1},\beta_{2}e_{1}e_{2},...,\beta_{2}e_{1}e_{2}...e_{2p-2},e_{1}\beta_{2},e_{2}e_{1}\beta_{2},...,e_{2p-2}e_{2p-3}...e_{1}\beta_{2},\\
&\alpha_{1}\beta_{1},\alpha_{2}\beta_{2},\alpha_{2}\beta_{2}e_{1},\alpha_{2}\beta_{2}e_{1}e_{2},...,\alpha_{2}\beta_{2}e_{1}e_{2}...e_{2p-2},e_{1}\alpha_{2}\beta_{2},e_{2}e_{1}\alpha_{2}\beta_{2},...,e_{2p-2}e_{2p-3}...e_{1}\alpha_{2}\beta_{2}\}
\end{align*}
\end{prop}

\titleformat{\section}{\normalfont\Large\bfseries}{\appendixname~\thesection.}{1em}{}
\begin{appendices}
	\numberwithin{equation}{section}
\section{Combinatorial Relations.}
There are a number of generalized relations for $\bar{U}_{q}(\mathfrak{sl}_{2})$ and its action on $X^{\otimes n}$, which we record here. The quantum group $\bar{U}_{q}(\mathfrak{sl}_{2})$ and its relations can be used to give the following generalized conditions:
\begin{align}
\Delta^{k}(K)&=K^{\otimes k+1} \label{eq:A1}\\
\Delta^{k}(E)&=\sum\limits_{i=0}^{k}(1^{\otimes i})\otimes E\otimes (K^{\otimes (k-i)}) \label{eq:A2}\\
\Delta^{k}(F)&=\sum\limits_{i=0}^{k}\big((K^{-1})^{\otimes i}\big)\otimes F\otimes(1^{\otimes (k-i)}) \label{eq:A3}\\
EF^{k}&=F^{k}E+(\frac{[k]}{q-q^{-1}})\big(q^{1-k}F^{k-1}K-q^{k-1}F^{k-1}K^{-1}\big) \label{eq:A4}\\
FE^{k}&=E^{k}F+(\frac{[k]}{q-q^{-1}})\big(q^{1-k}E^{k-1}K^{-1}-q^{k-1}E^{k-1}K\big) \label{eq:A5}\\
\Delta E^{k}&=\sum\limits_{i=0}^{k}\lambda_{i,k}E^{i}\otimes K^{i}E^{k-i} \label{eq:A6}\\
\Delta F^{k}&=\sum\limits_{i=0}^{k}\lambda_{i,k}K^{-i}F^{k-i}\otimes F^{i} \label{eq:A7}\\
\lambda_{i,k}&=q^{(i^{2}-ik)}\frac{([k]!)}{([i]!)([k-i]!)} \label{eq:A8}
\end{align}
The $\bar{U}_{q}(\mathfrak{sl}_{2})$ action on the basis elements satisfies the following:
\begin{align}
E^{n}\rho_{i_{1},...,i_{n},z}&=q^{\big(nz-\frac{1}{2}(n^{2}-n)-(\sum\limits_{j=1}^{n}i_{j})\big)}([n]!)x_{0,z} \label{eq:A9}\\
F^{z-n}\rho_{i_{1},...,i_{n},z}&=q^{\big(nz-\frac{1}{2}(n^{2}-n)-(\sum\limits_{j=1}^{n}i_{j})\big)}([z-n]!)x_{z,z} \label{eq:A10}\\
F^{k}x_{0,z}&=\sum\limits_{1\leq i_{j}\leq z}q^{\big(\frac{1}{2}(k^{2}+k)-(\sum\limits_{j=1}^{k}i_{j})\big)}([k]!)\rho_{i_{1},...,i_{k},z} \label{eq:A11}\\
E^{k}x_{z,z}&=\sum\limits_{1\leq i_{j}\leq z}q^{\big(\frac{1}{2}(z-k)(z-k+1)-(\sum\limits_{j=1}^{z-k}i_{j})\big)}([k]!)\rho_{i_{1},...,i_{z-k},z} \label{eq:A12}\\
E^{k}x_{z+1,z+1}&=[k](E^{k-1}x_{z,z})\otimes\nu_{0}+q^{-k}(E^{k}x_{z,z})\otimes\nu_{1} \label{eq:A13}\\
&=q^{k-z-1}[k]\nu_{0}\otimes(E^{k-1}x_{z,z})+\nu_{1}\otimes(E^{k}x_{z,z}) \label{eq:A14}\\
F^{k}x_{0,z+1}&=(F^{k}x_{0,z})\otimes\nu_{0}+q^{k-z-1}[k](F^{k-1}x_{0,z})\otimes\nu_{1} \label{eq:A15}\\
&=q^{-k}\nu_{0}\otimes(F^{k}x_{0,z})+[k]\nu_{1}\otimes(F^{k-1}x_{0,z}) \label{eq:A16}
\end{align}
These come from considering all contributions to the coefficient as different orderings of the integers $i_{1},...,i_{n}$, where each ordering describes the order in which the zeroes appeared. For the standard ordering with $i_{1}<i_{2}<...<i_{n}$, its contribution to the coefficient is just $q^{-z+i_{1}}q^{-z+i_{2}}...q^{-z+i_{n}}=q^{(-nz+\sum\limits_{j=1}^{n}i_{j})}$. Interchanging two integers in the ordering multiplies this by $q^{\pm 2}$, and the coefficient comes from considering all possible permutations.\\
\\
For  integers $1\leq i_{1}<i_{2}<...<i_{n}\leq z$, we have:
\begin{align}
\xi_{n,z}:=&\sum\limits_{1\leq i_{j}\leq z}q^{-2(\sum\limits_{j=1}^{n}i_{j})}=q^{-n-nz}\frac{([z]!)}{([n]!)([z-n]!)} \label{eq:A17}\\
\xi_{n,z}=& q^{-2z}\xi_{n-1,z-1}+\xi_{n,z-1}
\end{align}
where the recurrence relation comes from considering the two cases in $\xi_{n,z}$, when $i_{n}=z$ and when $i_{n}\neq z$.  
\end{appendices}
\subsection*{Acknowledgements}
The author's research was supported by EPSRC DTP grant EP/K502819/1. The author would like to thank the Isaac Newton Institute for Mathematical Sciences, Cambridge, for support and hospitality during the programme Operator Algebras: Subfactors and their applications, where work on this paper was undertaken. This work was supported by EPSRC grant no EP/K032208/1. The author thanks David Evans and Azat Gainutdinov for their comments.

\bibliographystyle{abbrv}

\end{document}